\documentclass[11pt, reqno]{amsart}
\pdfoutput=1

\usepackage[UKenglish]{babel}
\usepackage{enumitem}
\usepackage{amsmath, amssymb}
\usepackage{tikz-cd}
\usepackage{mathtools}
\usepackage{graphicx}
\usepackage{mathrsfs}
\usepackage{stmaryrd}
\usepackage{bm}
\usepackage{mathrsfs}
\usepackage[scr=euler,cal=dutchcal]{mathalfa}
\usepackage{stackengine}
\usepackage[hyphens]{url}
\usepackage[breaklinks=true]{hyperref}
\hypersetup{
    colorlinks = true,
    urlcolor = black,
    linkcolor= black,
    citecolor= black,
    filecolor= black,
}
\usepackage[normalem]{ulem}

\providecommand{\noopsort}[1]{}

\newcommand{\rotateRPY}[4][0/0/0]% point to be saved to \savedxyz, roll, pitch, yaw
{   \pgfmathsetmacro{\rollangle}{#2}
    \pgfmathsetmacro{\pitchangle}{#3}
    \pgfmathsetmacro{\yawangle}{#4}

    \pgfmathsetmacro{\newxx}{cos(\yawangle)*cos(\pitchangle)}% a
    \pgfmathsetmacro{\newxy}{sin(\yawangle)*cos(\pitchangle)}% d
    \pgfmathsetmacro{\newxz}{-sin(\pitchangle)}% g
    \path (\newxx,\newxy,\newxz);
    \pgfgetlastxy{\nxx}{\nxy};

    \pgfmathsetmacro{\newyx}{cos(\yawangle)*sin(\pitchangle)*sin(\rollangle)-sin(\yawangle)*cos(\rollangle)}% b
    \pgfmathsetmacro{\newyy}{sin(\yawangle)*sin(\pitchangle)*sin(\rollangle)+ cos(\yawangle)*cos(\rollangle)}% e
    \pgfmathsetmacro{\newyz}{cos(\pitchangle)*sin(\rollangle)}% h
    \path (\newyx,\newyy,\newyz);
    \pgfgetlastxy{\nyx}{\nyy};

    \pgfmathsetmacro{\newzx}{cos(\yawangle)*sin(\pitchangle)*cos(\rollangle)+ sin(\yawangle)*sin(\rollangle)}
    \pgfmathsetmacro{\newzy}{sin(\yawangle)*sin(\pitchangle)*cos(\rollangle)-cos(\yawangle)*sin(\rollangle)}
    \pgfmathsetmacro{\newzz}{cos(\pitchangle)*cos(\rollangle)}
    \path (\newzx,\newzy,\newzz);
    \pgfgetlastxy{\nzx}{\nzy};

    \foreach \x/\y/\z in {#1}
    {   \pgfmathsetmacro{\transformedx}{\x*\newxx+\y*\newyx+\z*\newzx}
        \pgfmathsetmacro{\transformedy}{\x*\newxy+\y*\newyy+\z*\newzy}
        \pgfmathsetmacro{\transformedz}{\x*\newxz+\y*\newyz+\z*\newzz}

    }
}

\tikzset{RPY/.style={x={(\nxx,\nxy)},y={(\nyx,\nyy)},z={(\nzx,\nzy)}}}

\makeatletter
\renewcommand{\tocsection}[3]{%
  \indentlabel{\@ifnotempty{#2}{\bfseries\ignorespaces#1 #2\quad}}\bfseries#3}
\renewcommand{\tocsubsection}[3]{%
  \indentlabel{\@ifnotempty{#2}{\ignorespaces#1 #2\quad}}#3}
\newcommand\@dotsep{4.5}
\def\@tocline#1#2#3#4#5#6#7{\relax
  \ifnum #1>\c@tocdepth % then omit
  \else
    \par \addpenalty\@secpenalty\addvspace{#2}%
    \begingroup \hyphenpenalty\@M
    \@ifempty{#4}{%
      \@tempdima\csname r@tocindent\number#1\endcsname\relax
    }{%
      \@tempdima#4\relax
    }%
    \parindent\z@ \leftskip#3\relax \advance\leftskip\@tempdima\relax
    \rightskip\@pnumwidth plus1em \parfillskip-\@pnumwidth
    #5\leavevmode\hskip-\@tempdima{#6}\nobreak
    \leaders\hbox{$\m@th\mkern \@dotsep mu\hbox{.}\mkern \@dotsep mu$}\hfill
    \nobreak
    \hbox to\@pnumwidth{\@tocpagenum{\ifnum#1=1\fi#7}}\par
    \nobreak
    \endgroup
  \fi}
\AtBeginDocument{%
\expandafter\renewcommand\csname r@tocindent0\endcsname{0pt}
}
\def\l@subsection{\@tocline{2}{0pt}{2.5pc}{5pc}{}}
\makeatother

\theoremstyle{plain}
\newtheorem{theorem}{Theorem}[section]
\newtheorem{lemma}[theorem]{Lemma}
\newtheorem{proposition}[theorem]{Proposition}
\newtheorem{corollary}[theorem]{Corollary}

\newtheorem*{claim*}{Claim}

\theoremstyle{definition}
\newtheorem{definition}[theorem]{Definition}
\newtheorem{remark}[theorem]{Remark}
\newtheorem{assumption}[theorem]{Assumption}
\newtheorem{example}[theorem]{Example}

\newcommand{\op}{\mathrm{op}} % opposite category 
\newcommand{\colim}
{\operatornamewithlimits{colim}}
\DeclareMathOperator{\Lan}{Lan}
\DeclareMathOperator{\ind}{Ind} % Ind-completion
\DeclareMathOperator{\pro}{Pro} % Pro-completion
\newcommand{\Pk}{\mathbb{P}_k} % Pebbling comonad P_k

\newcommand{\down}{{\downarrow}} % Downward closure
\newcommand{\pit}{\pitchfork}
\newcommand{\sk}[1]{#1^{\bullet}} % Stirling kernel
\newcommand{\N}{\mathbb{N}} % Natural numbers
\newcommand{\fp}{\mathrm{fp}} % Finitely presentable
\newcommand{\fg}{\mathrm{fg}} % Finitely generated
\newcommand{\fin}{\mathrm{fin}} % Finite
\newcommand{\R}{\mathscr{R}} % Relational structures
\newcommand{\cvr}{\prec} % Covering relation for forests
\newcommand{\V}{\mathscr{V}} % Variety of universal algebras
\newcommand{\CL}{\mathcal L} % Logic with counting quantifiers
\renewcommand{\k}{\mathbf{k}} % {1,...,k}
\newcommand{\woeq}{\text{\small{$\overset{w.o.}=$}}}
\newcommand{\w}{\widehat}
\newcommand{\yo}{\mathcal{y}}

\newcommand{\emb}{\rightarrowtail} % Embeddings
\newcommand{\into}{\hookrightarrow} % Inclusions (e.g., inclusion functor)
\newcommand{\epi}{\twoheadrightarrow} % Quotients
\newcommand{\id}{\mathrm{id}} % Identity morphisms
\newcommand{\Q}{\mathscr{Q}} % Left class in the factorisation system
\newcommand{\M}{\mathscr{M}} % Right class in the factorisation system

\tikzcdset{
    relay arrow/.default = 10pt,
    relay arrow/.style = {
        rounded corners,
        to path = {
            \pgfextra{
                \def\sourcecoordinate{\pgfpointanchor{\tikztostart}{center}}
                \def\targetcoordinate{\pgfpointanchor{\tikztotarget}{center}}
                \pgfmathanglebetweenpoints{\sourcecoordinate}{\targetcoordinate}
                \edef\tempangle{\pgfmathresult}
                \pgftransformrotate{\tempangle}
                \pgfmathifthenelse{#1>0}{\tempangle+90}{\tempangle-90}
                \pgfcoordinate{tempcoord}{\pgfpointanchor{\tikztostart}{\pgfmathresult}}
            }
            (tempcoord)
            -- ([yshift=#1]tempcoord)
            -- ([yshift=#1]tempcoord-|\tikztotarget.center)\tikztonodes
            --(\tikztotarget)
        }
    }
}

\renewcommand{\epsilon}{\varepsilon}
\renewcommand{\theta}{\vartheta}
\renewcommand{\phi}{\varphi}

\DeclareMathOperator{\Set}{\mathbf{Set}} % Sets
\DeclareMathOperator{\FinSet}{\mathbf{FinSet}} % Finite Sets
\DeclareMathOperator{\Stone}{\mathbf{Stone}} % Stone Spaces
\DeclareMathOperator{\Boole}{\mathbf{Boole}} % Boolean algebras
\DeclareMathOperator{\A}{\mathscr{A}} % Generic category
\DeclareMathOperator{\C}{\mathscr{C}} % Generic category
\DeclareMathOperator{\D}{\mathscr{D}} % Generic category
\DeclareMathOperator{\T}{\mathscr{T}} % Category of trees
\newcommand{\EM}{\mathbf{EM}} % Eilenberg-Moore category
\newcommand{\K}{\mathscr{K}} % KH topological algebras

\title{Polyadic Sets and Homomorphism Counting}

\author{Luca Reggio}
\address{Department of Computer Science, University of Oxford, United Kingdom}
\email{luca.reggio@cs.ox.ac.uk}

\thanks{Research supported by the European Union's Horizon 2020 research and innovation programme under the Marie Sk{\l}odowska-Curie grant agreement No 837724, and partially supported by the European Research Council (ERC) under the European Union’s
Horizon 2020 research and innovation program (grant agreement No 670624).}

\begin{document}

\maketitle

\begin{abstract}
A classical result due to Lov\'{a}sz (1967) shows that the isomorphism type of a graph is determined by homomorphism counts. That is, graphs $G$ and $H$ are isomorphic whenever the number of homomorphisms $K\to G$ is the same as the number of homomorphisms $K\to H$ for all graphs $K$. Variants of this result, for various classes of finite structures, have been exploited in a wide range of research fields, including graph theory and finite model theory. 

We provide a categorical approach to homomorphism counting based on the concept of polyadic (finite) set. The latter is a special case of the notion of polyadic space introduced by Joyal (1971) and related, via duality, to Boolean hyperdoctrines in categorical logic. We also obtain new homomorphism counting results applicable to a number of infinite structures, such as finitely branching trees and profinite algebras.
\end{abstract}

\tableofcontents

%%%%%%%%%%%%%%%%%%%%%%%%%%%%%%%
\section{Introduction}
A celebrated result of Lov\'{a}sz~\cite{Lovasz1967} states that two (finite, directed) graphs $G$ and $H$ are isomorphic if and only if, for all graphs $K$, the number of graph homomorphisms $K\to G$ is the same as the number of graph homomorphisms $K\to H$.\footnote{Lov\'{a}sz proved this result, more generally, for finite $\sigma$-structures over a finite relational signature $\sigma$.} This is a prototypical example of homomorphism counting result. Extensions and variations of this theorem have been exploited in a broad range of research areas, such as graph theory, finite model theory and quantum information~\cite{dvovrak2010recognizing,grohe2020counting,manvcinska2020quantum}.

Categorical generalisations of Lov\'{a}sz' result were proved by Pultr~\cite{pultr1973isomorphism} and Isbell~\cite{isbell1991some}, cf.\ also~\cite{Lovasz1972}, utilising a direct generalisation of Lov\'{a}sz' combinatorial counting argument. A different approach, which is generalised in the present paper, was adopted in a recent joint work with Dawar and Jakl~\cite{DJR2021}. These results apply to a large class of \emph{locally finite} categories, i.e.\ categories with only finitely many morphisms between any two objects. Examples include graphs, finite groups and, more generally, finite members of any equationally defined class of universal algebras.

The aim of this article is two-fold. On the one hand, we show that homomorphism counting results are basically a consequence of a more general result about polyadic (finite) sets, a special case of Joyal's polyadic spaces~\cite{Joyal1971}. The latter are dual---in the sense of Stone duality for Boolean algebras~\cite{Stone1936}---to Boolean hyperdoctrines, a fundamental tool of categorical logic~\cite{Lawvere2006}; for more details, see~\cite{Marques2021}. This exposes a new connection between homomorphism counting, as studied in graph theory and related fields, and the structural methods of category theory and categorical logic.

On the other hand, we exploit the framework of polyadic sets, combined with a topological argument, to extend homomorphism counting results beyond categories of finite structures. We thus obtain new general results that apply, for example, to finitely branching trees and (topologically finitely generated) profinite algebras. An application to finite-variable logics in finite model theory, which relies on the notion of game comonad introduced by Abramsky, Dawar \emph{et al.}, is also presented.

The present paper is structured as follows. In Section~\ref{s:prelim} we recall some basic categorical definitions and facts, and introduce the notion of (left- and right-) combinatorial category, in which the isomorphism types of objects are determined by homomorphism counts. In Section~\ref{s:polyadic-sets} we study polyadic (finite) sets and prove some of their main properties. These results are used in Section~\ref{s:hom-count-finite} to establish homomorphism counting results for locally finite categories. In Section~\ref{s:beyond-loc-finite} we prove a general theorem characterising the isomorphism type of certain objects in categories that are not necessarily locally finite, and specialise it to the setting of locally finitely presentable categories. Several applications of these results are discussed in Section~\ref{s:examples}.

%%%%%%%%%%%%%%%%%%%%%%%%%%%%%%%
\section{Preliminaries and Basic Notions}\label{s:prelim}
Throughout the paper, we assume the reader is familiar with basic notions of category theory; standard references include~\cite{AHS1990, MacLane1998}.

Recall that a category $\A$ is \emph{well-powered} if, for each object $a\in \A$, the collection of subobjects of $a$ is a set (as opposed to a proper class), see e.g.\ \cite[Definition~7.82]{AHS1990}. Dually, $\A$ is \emph{well-copowered} if the opposite category $\A^\op$ is well-powered.

\begin{assumption}\label{assump:well-(co)powered}
All categories $\A$ under consideration are assumed to be locally small, well-powered and well-copowered.
\end{assumption}

Let $\Set$ be the category of sets and functions. Given a category $\A$, denote by $\w{\A}$ the category of presheaves over $\A$, i.e.\ the functor category $\Set^{\A^\op}$.
The Yoneda embedding 
\[
\yo^{\A}\colon \A \to \w{\A}
\] 
sends an object $a$ of $\A$ to the presheaf 
\[
\yo^{\A}_a\colon \A^\op \to \Set, \quad b\mapsto \yo^{\A}_a(b)\coloneqq \hom_{\A}(b,a).
\]
When the category $\A$ is clear from the context, we omit the superscript and simply write $\yo$ and $\yo_a$.
A well known consequence of the Yoneda Lemma (cf.\ e.g.\ \cite[Theorem~6.20]{AHS1990}) states that, for all objects $a,b\in\A$, 
\begin{equation}\label{eq:Yoneda-reflects-isos}
a\cong b \enspace \Longleftrightarrow \enspace \yo_a \cong \yo_b.
\end{equation}
That is, $a$ and $b$ are isomorphic (as objects of $\A$) precisely when the hom-functors $\yo_a$ and $\yo_b$ are \emph{naturally} isomorphic. 

In general, two functors $F,G\colon \A\to\mathscr{B}$ can be pointwise isomorphic (i.e., $F(a)\cong G(a)$ for all $a\in \A$) without being naturally isomorphic; we illustrate this point with an example due to Joyal \cite[Exemple~6]{Joyal1981}.

\begin{example}
Let $\A$ be the category of finite sets and bijections. Consider the functor $F\colon \A\to\A$ sending a finite set $a$ to its set of permutations. Given a morphism $f\colon a\to b$ in $\A$, $F(f)$ sends a permutation $\phi$ of $a$ to the permutation $f\circ \phi\circ f^{-1}$ of $b$. Further, let $G\colon \A\to\A$ be the functor which sends a finite set to the set of its linear orderings. If $f\colon a\to b$ is a morphism in $\A$, $G(f)$ sends a linear order $R\subseteq a\times a$ on $a$ to the linear order $\{(f(x),f(y))\mid (x,y)\in R\}$ on $b$.

The functors $F$ and $G$ are pointwise isomorphic because the number of permutations of an $n$-element set $a$ is $n!$, which is also the number of linear orders on $a$. However, they are not naturally isomorphic; just note that there is no natural way of assigning a linear order to the identity permutation.
\end{example}

Broadly speaking, a homomorphism counting result asserts that the existence of a natural isomorphism $\yo_a \cong \yo_b$ in equation~\eqref{eq:Yoneda-reflects-isos} can be replaced by the weaker condition that $\yo_a$ and $\yo_b$ are \emph{pointwise} isomorphic.
\begin{definition}
A category $\A$ is said to be \emph{right-combinatorial} if, for all objects $a,b\in \A$, 
\[
a\cong b \enspace \Longleftrightarrow \enspace \forall c\in \A \ (\yo_a(c) \cong \yo_b(c)).
\]
\end{definition}

Note that the isomorphisms on the right-hand side above are nothing but bijections between sets. Thus, the isomorphism type of an object $a$ of a right-combinatorial category $\A$ is completely determined by the cardinalities $|\hom_{\A}(c,a)|$ of the hom-sets for $c$ which varies in $\A$.

There is of course a dual notion of left-combinatorial category: 
\begin{definition}
A category $\A$ is \emph{left-combinatorial} provided that $\A^\op$ is right-combinatorial. A category that is both left- and right-combinatorial is called a \emph{combinatorial category}.
\end{definition}

Let us hasten to point out that the term \emph{combinatorial category} was coined by Pultr~\cite{pultr1973isomorphism} (and used, e.g., in~\cite{DJR2021}) to refer to what we call a right-combinatorial category. As both notions of left- and right-combinatorial category are interesting and useful, we opted for a nomenclature that distinguishes between these dual notions.

%%%%%%%%%%%%%%%%%%%%%%%%%%%%%%%%%%%%
\section{Polyadic Sets}\label{s:polyadic-sets}

The aim of this section is to introduce the main tool of this paper: polyadic sets. In Section~\ref{s:amalgamation} we introduce the amalgamation property for presheaves. Polyadic sets, which are presheaves with the amalgamation property, are studied in Section~\ref{s:polyad-spaces-SK}, where the construction of the Stirling kernel is presented. In Section~\ref{s:pointwise-isomorphisms} we prove a result concerning Stirling kernels that will be used to establish homomorphism counting results in the subsequent sections.

\subsection{The amalgamation property}\label{s:amalgamation}
Let us say that a category $\A$ has the \emph{amalgamation property} if any span of morphisms in $\A$ can be completed to a commutative square:
\[\begin{tikzcd}
{\cdot} \arrow{r} \arrow{d} & {\cdot} \arrow[dashed]{d} \\
{\cdot} \arrow[dashed]{r} & {\cdot}
\end{tikzcd}\]
By extension, we say that a presheaf $F\colon \A^\op\to \Set$ has the amalgamation property if so does its \emph{category of elements} $\int F$. Recall that objects of $\int F$ are pairs $(a,x)$ with $a\in \A$ and $x\in F(a)$, and a morphism $(a,x)\to (a',x')$ is a morphism $f\colon a\to a'$ in $\A$ such that $F(f)(x')=x$ (see, e.g., \cite[\S I.5]{MM1994}). Unravelling the definition above, we see that a presheaf $F\colon \A^\op\to\Set$ has the amalgamation property if, for all spans of morphisms 
\[\begin{tikzcd}
b & {\cdot} \arrow{l}[swap]{g} \arrow{r}{f} & a
\end{tikzcd}\]
in $\A$ and elements $x\in F(a)$ and $y\in F(b)$ such that $F(f)(x)=F(g)(y)$, there exist a commutative square
\[\begin{tikzcd}
{\cdot} \arrow{r}{f} \arrow{d}[swap]{g} & a \arrow[dashed]{d}{g'} \\
b \arrow[dashed]{r}{f'} & c
\end{tikzcd}\]
in $\A$ and $z\in F(c)$ such that $F(g')(z)=x$ and $F(f')(z)=y$.

\begin{lemma}\label{l:repr-polyadic-set}
Any representable functor $\A^\op\to\Set$ has the amalgamation property.
\end{lemma}
\begin{proof}
It suffices to show that, for an arbitrary object $a\in \A$, the functor $\yo_a\colon \A^\op\to\Set$ has the amalgamation property. The statement then follows because the amalgamation property for presheaves is preserved under natural isomorphisms.

Consider a span of morphisms
\[\begin{tikzcd}
b & {\cdot} \arrow{l}[swap]{g} \arrow{r}{f} & c
\end{tikzcd}\]
in $\A$ and let $h\in \yo_a(c)$ and $k\in \yo_a(b)$ satisfy $\yo_a(f)(h)=\yo_a(g)(k)$. The latter equation amounts to the commutativity of the following square:
\[\begin{tikzcd}
{\cdot} \arrow{r}{f} \arrow{d}[swap]{g} & c \arrow{d}{h} \\
b \arrow{r}{k} & a
\end{tikzcd}\]
Observe that the identity morphism $\id_a\in \yo_a(a)$ satisfies $\yo_a(h)(\id_a)=h$ and $\yo_a(k)(\id_a)=k$, and so $\yo_a$ has the amalgamation property.
\end{proof}

When the category $\A$ admits pushouts, the amalgamation property for presheaves $\A^\op\to\Set$ can be rephrased as follows. Recall that a commutative square
\[\begin{tikzcd}
u \arrow{r}{j'} \arrow{d}[swap]{i'} & s \arrow{d}{i} \\ 
t \arrow{r}{j} & {\cdot}
\end{tikzcd}\]
in $\Set$ is a \emph{quasi-pullback} if the unique mediating morphism $u\to p$, where $p$ is the pullback of $i$ along $j$, is a surjection. In other words, whenever $x\in s$ and $y\in t$ satisfy $i(x)=j(y)$, there is a (not necessarily unique) $z\in u$ such that $j'(z)=x$ and $i'(z)=y$. 

We have the following elementary result:
\begin{lemma}\label{l:amalg-quasi-pull}
Let $\A$ be a category with pushouts. The following statements are equivalent for any functor $F\colon \A^\op \to\Set$:
\begin{enumerate}
\item $F$ has the amalgamation property.
\item $F$ sends pushout squares in $\A$ to quasi-pullbacks in $\Set$.
\end{enumerate}
\end{lemma}

Note that, whenever $\A$ has pushouts, Lemma~\ref{l:repr-polyadic-set} can be deduced from Lemma~\ref{l:amalg-quasi-pull} and the fact that representable functors preserve limits. Further, in this case any presheaf $F\colon \A^\op\to\Set$ that is a filtered colimit of representable ones has the amalgamation property. Just recall that filtered colimits in $\Set$ commute with finite limits (see e.g.\ \cite[Theorem~1 p.~215]{MacLane1998}).

\subsection{Polyadic sets and Stirling kernels}\label{s:polyad-spaces-SK}

\begin{definition}\label{def:polyadic-set}
A \emph{polyadic set} on a category $\A$ is a functor \[F\colon \A^\op\to\Set\] with the amalgamation property. If, in addition, the sets $F(a)$ are finite for all $a\in\A$, then $F$ is called a \emph{polyadic finite set}. 
\end{definition}

Let $\FinSet$ denote the full subcategory of $\Set$ defined by the finite sets. Throughout this paper, we shall identify polyadic finite sets with functors ${\A^\op\to\FinSet}$ satisfying the amalgamation property.

Note that, with this terminology, Lemma~\ref{l:repr-polyadic-set} states that all representable presheaves are polyadic sets. 

\begin{remark}\label{rem:amalgam-prop-equiv}
Polyadic sets are a discrete variant of Joyal's notion of polyadic space~\cite{Joyal1971}, which we now recall. Let $\Stone$ denote the category of \emph{Stone spaces} (i.e., zero-dimensional compact Hausdorff spaces, also known as \emph{Boolean spaces}) and continuous maps. In~\cite{Joyal1971}, a \emph{polyadic space} (on $\FinSet$) is defined to be a functor 
\[
F\colon \FinSet^\op\to\Stone
\]
satisfying the following conditions:
\begin{enumerate}[label=(\roman*)]
\item $F$ sends pushout squares in $\FinSet$ to quasi-pullbacks\footnote{A commutative square in $\Stone$ is a quasi-pullback if its image under the underlying-set functor $\Stone\to \Set$ is a quasi-pullback in $\Set$.} in $\Stone$.
\item $F(f)$ is an open map for any morphism $f$ in $\FinSet$.
\end{enumerate}
The observation that condition~(i) above can be generalised in terms of the amalgamation property is due to Marqu\`{e}s~\cite{Marques2021}.

By Stone duality between $\Stone$ and the category $\Boole$ of Boolean algebras and their homomorphisms~\cite{Stone1936}, polyadic spaces can be identified with the pointwise duals of Boolean hyperdoctrines ${\FinSet\to \Boole}$. For a thorough account of this connection, we refer the interested reader to~\cite{Marques2021}.

Polyadic sets are obtained by dispensing with condition~(ii) above, which has a topological nature. In particular, note that a polyadic finite set (on $\FinSet$) is the same as a polyadic space $F\colon \FinSet^\op\to\Stone$ such that, for all $n\in \FinSet$, $F(n)$ is a finite (equivalently, discrete) space.
\end{remark}

For the remainder of this section, we fix an arbitrary category $\A$ equipped with a \emph{proper factorisation system} $(\Q,\M)$. That is, a weak factorisation system such that every $\Q$-morphism is an epimorphism and every $\M$-morphism is a monomorphism. For more details and some basic properties of these factorisation systems, see Appendix~\ref{s:fact-systems}. We refer to $\M$-morphisms as \emph{embeddings} and denote them by $\emb$. $\Q$-morphisms will be referred to as \emph{quotients} and denoted by $\epi$. A \emph{proper quotient} is a quotient that is not an isomorphism (in view of Lemma~\ref{l:factorisation-properties}\ref{isos} in the appendix, proper quotients can be identified with arrows in $\Q\setminus \M$).

The following easy observation will come in handy in the following.
\begin{lemma}\label{l:quot-inj-polyadic}
Any polyadic set $F\colon \A^\op\to\Set$ sends quotients in $\A$ to injections.
\end{lemma}
\begin{proof}
Let $f\colon n\epi m$ be a quotient in $\A$ and let $s,t\in F(m)$ be arbitrary elements satisfying $F(f)(s)=F(f)(t)$. By the amalgamation property for $F$ there are a commutative diagram
\[\begin{tikzcd}
n \arrow[twoheadrightarrow]{r}{f} \arrow[twoheadrightarrow]{d}[swap]{f} & m \arrow[dashed]{d}{h}\\
m \arrow[dashed]{r}{g} & p
\end{tikzcd}\]
and an element $u\in F(p)$ such that $F(h)(u)=s$ and $F(g)(u)=t$. Since $g\circ f=h\circ f$ and $f$ is an epimorphism, we get $g=h$. Thus, $s=F(g)(u)=t$ and so $F(f)$ is injective.
\end{proof}

Next, we introduce the \emph{Stirling kernel} construction for polyadic sets. This notion is due to Joyal in the more general context of polyadic spaces\footnote{Private e-mail communication.} (see Par\'e's work~\cite{Pare1999} for a special case) and has been further explored by Marqu\`{e}s in~\cite{Marques2021}.

Before giving a formal definition, let us motivate the idea underlying Stirling kernels with an example. Consider the polyadic set
\[
\FinSet^\op \to \Set, \quad n\mapsto [0,1]^n
\]
obtained by restricting the representable functor $\yo_{[0,1]}\colon \Set^\op\to \Set$, where $[0,1]$ is the unit interval regarded as a set. Each (generalised) cube $[0,1]^n$ can be decomposed into smaller `pieces'. Some of these pieces will have maximal dimension, i.e.\ dimension $n$, while the others will have lower dimension. The decomposition of $[0,1]$ is the trivial one consisting of a single piece, namely $[0,1]$ itself (dimension~$1$). Moving one dimension up, the square $[0,1]^2$ can be decomposed into its diagonal (dimension~$1$) and two triangles (dimension~$2$):

  
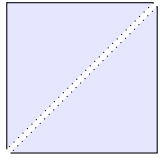
\[\begin{tikzpicture}
   % \draw[color=black] (0,0) -- (2,2);
    \draw[fill=blue!10] (0.05,0) -- (2,0) -- (2,1.95);
     \draw[style=dotted] (2,1.95) -- (0.05,0);
     \draw[fill=blue!10] (0,0.05) -- (0,2) -- (1.95,2);
     \draw[style=dotted] (1.95,2) -- (0,0.05);
  \end{tikzpicture}\]

\noindent Similarly, $[0,1]^3$ can be decomposed into its diagonal, consisting of the points $(x,y,z)$ such that $x=y=z$, the six triangles depicted below consisting of the points where exactly two of the coordinates are equal,

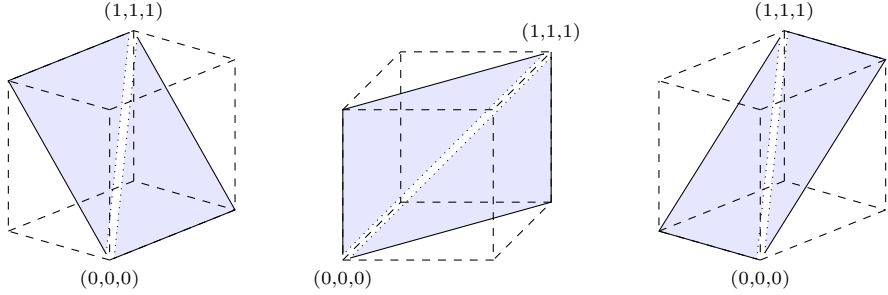
\begin{equation*}
\begin{tikzpicture}\rotateRPY{0}{60}{0}
  \begin{scope}[RPY]
    \draw[fill=blue!10] (0,0.07,-0.07) -- (0,2,-2) -- (1.93,2,-2); 
    \draw[style=dotted] (1.93,2,-2) -- (0,0.07,-0.07);
        \draw[fill=blue!10] (0.07,0,0) -- (2,0,0) -- (2,1.93,-1.93); 
        \draw[style=dotted] (2,1.93,-1.93) -- (0.07,0,0);
    \draw[style=dashed]  (0,0,0) -- (2,0,0) -- (2,2,0) -- (0,2,0) -- cycle;
    \draw[style=dashed] (0,2,0) -- (0,2,-2) -- (2,2,-2) -- (2,2,0);
    \draw[style=dashed] (2,2,-2) -- (2,0,-2) -- (2,0,0);
    \draw[style=dashed] (0,0,0) -- (0,0,-2) -- (0,2,-2);
    \draw[style=dashed] (0,0,-2) -- (2,0,-2);
    \draw (0,0,0) node[below] {{\tiny(0,0,0)}};
    \draw (2,2,-2) node[above] {{\tiny(1,1,1)}};
    \end{scope}
  \end{tikzpicture}
  \ \ \ \ \ \ \
  \begin{tikzpicture}
    \draw[fill=blue!10] (0.07,0,-0.07) -- (2,0,-2) -- (2,1.93,-2); 
    \draw[style=dotted] (2,1.93,-2) -- (0.07,0,-0.07);
        \draw[fill=blue!10] (0,0.07,0) -- (0,2,0) -- (1.93,2,-1.93); 
        \draw[style=dotted] (1.93,2,-1.93) -- (0,0.07,0);
    \draw[style=dashed] (0,0,0) -- (2,0,0) -- (2,2,0) -- (0,2,0) -- cycle;
    \draw[style=dashed] (0,2,0) -- (0,2,-2) -- (2,2,-2) -- (2,2,0);
    \draw[style=dashed] (2,2,-2) -- (2,0,-2) -- (2,0,0);
    \draw[style=dashed] (0,0,0) -- (0,0,-2) -- (0,2,-2);
    \draw[style=dashed] (0,0,-2) -- (2,0,-2);
        \draw (0,0,0) node[below] {{\tiny(0,0,0)}};
    \draw (2,2,-2) node[above] {{\tiny(1,1,1)}};
  \end{tikzpicture}
  \ \ \ \ \ \ \
  \begin{tikzpicture}\rotateRPY{0}{60}{0}
  \begin{scope}[RPY]
    \draw[fill=blue!10] (0.07,0.07,0) -- (2,2,0) -- (2,2,-1.93); 
    \draw[style=dotted] (2,2,-1.93) -- (0.07,0.07,0);
        \draw[fill=blue!10] (0,0,-0.07) -- (0,0,-2) -- (1.93,1.93,-2); 
        \draw[style=dotted] (1.93,1.93,-2) -- (0,0,-0.07);
    \draw[style=dashed] (0,0,0) -- (2,0,0) -- (2,2,0) -- (0,2,0) -- cycle;
    \draw[style=dashed] (0,2,0) -- (0,2,-2) -- (2,2,-2) -- (2,2,0);
    \draw[style=dashed] (2,2,-2) -- (2,0,-2) -- (2,0,0);
    \draw[style=dashed] (0,0,0) -- (0,0,-2) -- (0,2,-2);
    \draw[style=dashed] (0,0,-2) -- (2,0,-2);
        \draw (0,0,0) node[below] {{\tiny(0,0,0)}};
    \draw (2,2,-2) node[above] {{\tiny(1,1,1)}};
    \end{scope}
  \end{tikzpicture}
  \end{equation*}
along with the remaining three-dimensional (convex) polyhedra. (All these decompositions are determined by the polyadic set in a canonical way.) The Stirling kernel of this polyadic set associates with a finite set $n$ the subset of $[0,1]^n$ defined by the `generic' points, i.e.\ those points that belong to a piece of maximal dimension in the decomposition; equivalently, the points whose coordinates are all distinct. The discussion above for $n\leq 3$ suggests that in the passage from the polyadic set to its Stirling kernel no essential information is lost: a cube $[0,1]^n$ can be fully reconstructed as long as we know the generic points of all cubes $[0,1]^m$ with~$m\leq n$.

\begin{definition}
Given a polyadic set $F\colon \A^\op\to\Set$ and an object $n\in\A$, we say that an element $x\in F(n)$ is \emph{degenerate} if there exist a proper quotient $f\colon n\epi m$ and an element $y\in F(m)$ such that $F(f)(y)=x$. The elements of $F(n)$ that are not degenerate are called \emph{generic}. The subset of $F(n)$ consisting of the generic elements is denoted by $\sk{F}(n)$. 
\end{definition}

If $g$ is a morphism in $\A$, then the function $F(g)$ need not preserve generic elements. However, this is the case when $g$ is an embedding. This is the content of the following lemma, where we denote by $\A_*$ the category with the same objects as $\A$ and morphisms the embeddings in $\A$. (The dual of the category $\A_*$ is denoted by $\A_*^\op$, in place of the unwieldy $(\A_*)^\op$.)

\begin{lemma}
Let $F\colon \A^\op\to \Set$ be a polyadic set. The assignment \[{n\mapsto \sk{F}(n)}\] yields a functor $\sk{F}\colon \A_{*}^\op\to \Set$.
\end{lemma}
\begin{proof}
It suffices to show that, for all embeddings $g\colon n\emb n'$ in $\A$, the map $F(g)\colon F(n')\to F(n)$ preserves generic elements, for then it restricts to a map $\sk{F}(n')\to \sk{F}(n)$ (functoriality is clear). 

Let $x\in F(n')$ be an arbitrary generic element. Suppose that there exist a quotient $f\colon n\epi m$ and an element $y\in F(m)$ such that $F(f)(y)=F(g)(x)$. By Lemma~\ref{l:amalg-quotients} in the appendix, the span formed by $f$ and $g$ can be completed to a commutative square
\[\begin{tikzcd}
n \arrow[twoheadrightarrow]{r}{f} \arrow[rightarrowtail]{d}[swap]{g} & m \arrow[dashed]{d}{g'} \\
n' \arrow[dashed,twoheadrightarrow]{r}{f'} & p
\end{tikzcd}\]
and there exists $z\in F(p)$ satisfying $F(f')(z)=x$ (and $F(g')(z)=y$). As $f'$ is a quotient and $x$ is generic, $f'$ must be an isomorphism. It follows easily by Lemma~\ref{l:factorisation-properties}\ref{compositions},\ref{isos},\ref{cancellation-m} that $f$ is an isomorphism. Therefore, $F(g)(x)$ is a generic element of $F(n)$. 
\end{proof}

\begin{definition}
The \emph{Stirling kernel} of a polyadic set $F\colon \A^\op\to\Set$ is the functor $\sk{F}\colon \A_{*}^\op\to \Set$.
\end{definition}
In Corollary~\ref{cor:kernel-is-polyadic} below we shall see that, under mild assumptions on $\A$, the Stirling kernel of a polyadic set on $\A$ is a polyadic set on $\A_{*}$.

The next example will play a pivotal role in the study of homomorphism counting results.
\begin{example}\label{ex:kernel-repr}
Let $F\colon \A^\op\to\Set$ be any representable presheaf. Up to a natural isomorphism, we can assume that $F=\yo_a$ for some object $a\in \A$. By Lemma~\ref{l:repr-polyadic-set}, $F$ is a polyadic set. Its Stirling kernel $\sk{F}\colon \A_*^\op\to\Set$ sends an object $b\in \A$ to the set $\M(b,a)$ of all embeddings $b\emb a$. In other words, an element of $F(b)=\yo_a(b)$ is generic if, and only if, it is an embedding. This was proved in \cite[Lemma~11]{DJR2021}; for the sake of completeness, we provide a proof of this fact.

 Let $f$ be an arbitrary element of $\yo_a(b)$. Assume that $f$ is generic, and take its $(\Q,\M)$ factorisation:
    \[\begin{tikzcd}
        b \arrow[rr, relay arrow=2ex, "f"] \arrow[r, twoheadrightarrow, "g"] & c \arrow[r, rightarrowtail, "h"] & a
    \end{tikzcd}\]
Since $\yo_a(g)(h)=f$ and the latter is generic, $g$ must be an embedding. By Lemma~\ref{l:factorisation-properties}\ref{compositions}, $f=h\circ g$ is also an embedding. 
    
 Conversely, suppose that $f$ is an embedding and pick a quotient $g\colon b\epi c$ and an element $h\in \yo_a(c)$ such that $\yo_a(g)(h)=f$, i.e.\ $h\circ g=f$. By Lemma~\ref{l:factorisation-properties}\ref{cancellation-m}, $g$ is an embedding, so it is not a proper quotient. It follows that $f$ is generic.
\end{example}
The next proposition shows that, under mild assumptions, a polyadic set can be recovered from its Stirling kernel (as an appropriate left Kan extension).\footnote{For applications to homomorphism counting in the next section, we will only need the first part of Proposition~\ref{p:reconstruction}. The reader unfamiliar with the concept of Kan extension can safely ignore the second part of the lemma.} 
This observation is due to Joyal (for polyadic spaces), see also~\cite{Marques2021}; the proof offered below is a simple adaptation to the discrete case.

For all objects $a\in \A$, we denote by $\Q(a)$ and $\M(a)$, respectively, the poset of quotients of $a$ and the poset of embeddings of $a$; for a definition, see Appendix~\ref{s:fact-systems}. Loosely speaking, the elements of $\Q(a)$ are equivalence classes of quotients from $a$, and the elements of $\M(a)$ are equivalence classes of embeddings into $a$. The category $\A$ is said to be \emph{$\Q$-well-founded} if, for all objects $a\in \A$, the poset $\Q(a)$ is well-founded. That is, any non-empty subset of $\Q(a)$ has a minimal element. Similarly, $\A$ is \emph{$\M$-well-founded} if, for all objects $a\in \A$, the poset $\M(a)$ is well-founded.
\begin{proposition}\label{p:reconstruction}
Let $F\colon \A^\op\to \Set$ be a polyadic set and assume that $\A$ is $\Q$-well-founded. The following statements hold:
\begin{enumerate}[label=(\alph*)]
\item\label{coprod-formula} For all $n\in \A$, \[F(n)\cong \coprod_{n\epi m} \sk{F}(m)\] where the coproduct is indexed by the set $\Q(n)$.
\item\label{Lan-formula} $F\cong \Lan_J \sk{F}$, where $\Lan_J \sk{F}$ is the left Kan extension of $\sk{F}$ along the inclusion functor $J\colon \A_*^\op\into \A^\op$.
\end{enumerate}
\end{proposition}
\begin{proof}
(a) Let $\A$ and $F\colon \A^\op\to\Set$ be as in the statement. For each quotient $f\colon n\epi m$ in $\Q(n)$, the function $F(f)\colon F(m)\to F(n)$ is an injection by Lemma~\ref{l:quot-inj-polyadic}. Let $\xi_f\colon \sk{F}(m)\emb F(n)$ be the obvious restriction of $F(f)$. By the universal property of the coproduct in $\Set$, the family of functions $\{\xi_f \mid f\in \Q(n)\}$ induces a unique map
\[
\xi\colon \coprod_{n\epi m} \sk{F}(m) \to F(n).
\]
We claim that $\xi$ is a bijection. For injectivity, suppose that there are quotients $f_1\colon n\epi m_1$, $f_2\colon n\epi m_2$ and elements $x_1\in \sk{F}(m_1)$ and $x_2\in \sk{F}(m_2)$ such that $\xi(x_1)=\xi(x_2)$. That is, $F(f_1)(x_1)=F(f_2)(x_2)$. By the amalgamation property for $F$ and Lemma~\ref{l:amalg-quotients}, there are quotients $g_1\colon m_1 \epi p$ and $g_2\colon m_2\epi p$ making the square below commute,
\[\begin{tikzcd}
n \arrow[twoheadrightarrow]{r}{f_1} \arrow[twoheadrightarrow]{d}[swap]{f_2} & m_1 \arrow[twoheadrightarrow, dashed]{d}{g_1} \\
m_2 \arrow[twoheadrightarrow, dashed]{r}{g_2} & p
\end{tikzcd}\]
and $y\in F(p)$ such that $F(g_1)(y)=x_1$ and $F(g_2)(y)=x_2$. As $x_1$ and $x_2$ are generic, $g_1$ and $g_2$ must be isomorphisms. In particular, the isomorphism $g_2^{-1}\circ g_1\colon m_1\to m_2$ witnesses the fact that $f_1=f_2$ as elements of $\Q(n)$. Since $\xi$ is injective on each summand, it follows that $x_1=x_2$. 

For surjectivity of $\xi$, suppose that $x\in F(n)$ and consider the set 
\[
S\coloneqq \{f\in \Q(n)\mid F(f)^{-1}(x)\neq \emptyset\}.
\] 
If $S=\emptyset$, then $x\in \sk{F}(n)$ and thus it belongs to the image of $\xi$. Therefore, assume that $S$ is non-empty. Because $\A$ is $\Q$-well-founded, $S$ has a minimal element $g\colon n\epi m$. Let $y\in F(m)$ be such that $F(g)(y)=x$. As $g$ is minimal, $y$ must belong to $\sk{F}(m)$, and so $x$ is in the image of $\xi$.

(b) Fix an arbitrary object $n\in \A$. Consider the composite functor
\[\begin{tikzcd}
G\colon J\down n \arrow{r}{\pi} & \A_*^\op \arrow{r}{\sk{F}} & \Set
\end{tikzcd}\] 
where $J\down n$ is the comma category and $\pi$ sends an object $(m, J(m)\to n)$, where $J(m)\to n$ is an arrow in $\A^{\op}$, to its first component $m$. In view of Lemma~\ref{l:amalg-quot}\ref{comma-cat-quot}, we can replace without loss of generality the category $J\down n$ with its full subcategory $\D$ consisting of the pairs $(m, J(m)\to n)$ whose second components correspond to quotients in $\A$. Note that $\D$ is a groupoid, i.e.\ every arrow in $\D$ is an isomorphism. Just observe that a morphism 
\[
(m',J(m')\to n)\to (m, J(m)\to n)
\]
in $\D$ corresponds to a commutative triangle
\[\begin{tikzcd}
{} & n \arrow[twoheadrightarrow]{dl} \arrow[twoheadrightarrow]{dr} & {} \\
m \arrow[rightarrowtail]{rr} & {} & m' 
\end{tikzcd}\]
in $\A$, and by Lemma~\ref{l:factorisation-properties}\ref{cancellation-e},\ref{isos} the horizontal arrow in necessarily an isomorphism.
Upon choosing representatives, this is equivalent to considering the small diagram
\[\begin{tikzcd}
G'\colon \Q(n)^\op \arrow{r}{\pi'} & \A_*^\op \arrow{r}{\sk{F}} & \Set
\end{tikzcd}\] 
where $\Q(n)$ is regarded as a set (i.e., a discrete category) and $\pi'$ sends (the equivalence class of) a quotient $n\epi m$ in $\A$ to its codomain.

The colimit of $G$ thus coincides with the coproduct $\coprod_{n\epi m} \sk{F}(m)$ indexed by the set $\Q(n)$. Moreover, any morphism $h\colon n\to n'$ in $\A$ induces a unique function
\[
\coprod_{n\epi m'} \sk{F}(m') \ \longleftarrow \ \coprod_{n'\epi m} \sk{F}(m)
\]
whose restriction to $\sk{F}(m)$, for all (representatives of equivalence classes of) quotients ${f\colon n'\epi m}$, is the function $\sk{F}(m)\to \sk{F}(m')$ obtained by applying $\sk{F}$ to the bottom horizontal arrow in the following commutative diagram.
\[\begin{tikzcd}
n \arrow{r}{h} \arrow[dashed,twoheadrightarrow]{d} & n' \arrow[twoheadrightarrow]{d}{f} \\
m' \arrow[dashed,rightarrowtail]{r} & m
\end{tikzcd}\]
By item (a), and the colimit formula for pointwise left Kan extensions (cf.\ \cite[Theorem~1 p.~237]{MacLane1998} for the dual statement), we see that $F\cong \Lan_J \sk{F}$.
\end{proof}

We illustrate the previous proposition by means of an example, which also justifies the terminology \emph{Stirling kernel}:
\begin{example}
Let $F\colon \FinSet^\op\to\Set$ be any polyadic set on the category $\FinSet$, and equip the latter with the usual (surjective, injective) factorisation system. We denote by $\uline{n}\in\FinSet$ an $n$-element set. For all non-negative integers $m\leq n$, the number of non-equivalent surjections $\uline{n}\epi \uline{m}$ in $\Q(\uline{n})$ coincides with the number of ways to partition an $n$-element set into $m$ non-empty subsets. This is commonly denoted by $S(n,m)$ and known as the \emph{Stirling number of the second kind} associated with the pair $(n,m)$, see e.g.\ \cite[\S 1.5]{Godsil1993}. Therefore, denoting by $k\, \uline{n}$ the disjoint sum of $k$ copies of $\uline{n}$, Proposition~\ref{p:reconstruction} yields the formula
\[
F(\uline{n})\cong \coprod_{m\leq n} S(n,m) \, \sk{F}(\uline{m})
\]
for all finite sets $\uline{n}$. In particular, if $F$ is a polyadic finite set, we get
\[
|F(\uline{n})|= \sum_{m\leq n} S(n,m) \cdot |\sk{F}(\uline{m})|.
\]
\end{example}

\begin{remark}
The assumption in Proposition~\ref{p:reconstruction} that $\A$ be $\Q$-well-founded is necessary, even when $F$ is a polyadic \emph{finite} set. For instance, let $\A$ be the set $\N$ of natural numbers with the usual total order, regarded as a category. If $\A$ is equipped with the factorisation system $(\Q,\M)$ where $\Q$ consists of all morphisms, and $\M$ consists of the identity morphisms, then $\A_*$ is the category with set of objects $\N$ and no arrows but identities. Note that $\A$ is not $\Q$-well-founded. Just observe that, for all $n\in \A$, the strictly increasing sequence of natural numbers $n < n+1 < n+2 < \cdots$ induces the strictly decreasing sequence
\[
n \epi n+1 \epi n+2 \epi \cdots
\]
in $\Q(n)$. Thus, the latter is not well-founded.
Now, let $a$ be any non-empty finite set, and let $F\colon \A^\op \to \FinSet$ be the constant functor of value $a$. It is not difficult to see that $F$ is a polyadic finite set, and ${\sk{F}\colon \A_*^\op\to\FinSet}$ is the constant functor of value $\emptyset$. Then ${\Lan_J \sk{F}\colon \A^\op\to \FinSet}$ is also the constant functor of value $\emptyset$ and so $F\not\cong \Lan_J \sk{F}$.
\end{remark}

\begin{corollary}\label{cor:kernel-is-polyadic}
Let $\A$ be $\Q$-well-founded. The Stirling kernel of a polyadic (finite) set on $\A$ is a polyadic (finite) set on $\A_*$.
\end{corollary}
\begin{proof}
Let $\A$ be as in the statement and let $F\colon \A^\op\to\Set$ be an arbitrary polyadic set. We claim that $\sk{F}\colon \A_*^\op\to\Set$ has the amalgamation property. Consider embeddings $f_1\colon n\emb m_1$, $f_2\colon n\emb m_2$ and generic elements $s\in \sk{F}(m_1)$ and $t\in\sk{F}(m_2)$ such that $F(f_1)(s)=F(f_2)(t)$. By the amalgamation property for $F$, there exist a commutative square
\[\begin{tikzcd}
n \arrow[rightarrowtail]{r}{f_1} \arrow[rightarrowtail]{d}[swap]{f_2} & m_1 \arrow[dashed]{d}{g_2} \\
m_2 \arrow[dashed]{r}{g_1} & p
\end{tikzcd}\]
in $\A$ and an element $u\in F(p)$ such that $F(g_2)(u)=s$ and $F(g_1)(u)=t$. By Proposition~\ref{p:reconstruction}\ref{coprod-formula}, there exist a quotient $h\colon p\epi q$ and a generic element $v\in \sk{F}(q)$ such that $F(h)(v)=u$. Clearly, we have 
\[
F(h\circ g_2)(v)=s \ \text{ and } \ F(h\circ g_1)(v)=t.
\] 
It suffices to show that $h\circ g_1$ and $h\circ g_2$ are embeddings, for then the following commutative square in $\A_*$ shows that $\sk{F}$ has the amalgamation property.
\[\begin{tikzcd}
n \arrow[rightarrowtail]{r}{f_1} \arrow[rightarrowtail]{d}[swap]{f_2} & m_1 \arrow[rightarrowtail, dashed]{d}{h\circ g_2} \\
m_2 \arrow[rightarrowtail, dashed]{r}{h\circ g_1} & q
\end{tikzcd}\]
To see that $h\circ g_1$ is an embedding, take its $(\Q,\M)$ factorisation:
\[\begin{tikzcd}[row sep=0.5em]
m_2 \arrow[twoheadrightarrow]{dr}[swap]{e} \arrow{rr}{h\circ g_1} & & q \\
{} & {\cdot} \arrow[rightarrowtail,shorten >=0.5ex]{ur}[swap]{l} & 
\end{tikzcd}\]
Let $x\coloneqq F(l)(v)$. As $F(e)(x)=t$ and $t$ is generic, it follows that $e$ is an isomorphism. Therefore, $h\circ g_1$ is an embedding by Lemma~\ref{l:factorisation-properties}\ref{compositions},\ref{isos}. The same proof, mutatis mutandis, shows that $h\circ g_2$ is an embedding.

Finally, note that if $F$ is a polyadic \emph{finite} set, then so is $\sk{F}$.
\end{proof}

\subsection{Pointwise isomorphisms}\label{s:pointwise-isomorphisms}
The main result of this section (Proposition~\ref{pr:pointwise-iso-kernel} below) states that, under appropriate assumptions, whenever two polyadic finite sets are pointwise isomorphic, their Stirling kernels are also pointwise isomorphic. This observation is at the core of the homomorphism counting results presented in Sections~\ref{s:hom-count-finite} and~\ref{s:beyond-loc-finite}. 

Recall that two parallel functors $F, G\colon \C\to \D$ are \emph{pointwise isomorphic} if, for all $c\in \C$, there is an isomorphism $\eta_c\colon F(c)\to G(c)$ in $\D$. This contrasts with the concept of \emph{natural isomorphism} between $F$ and $G$, whereby the isomorphisms $\eta_c$ are required to be natural in $c$.

In view of Proposition~\ref{p:reconstruction}, if $F$ and $G$ are polyadic sets whose Stirling kernels $\sk{F}$ and $\sk{G}$ are pointwise isomorphic, then $F$ and $G$ are also pointwise isomorphic. The converse holds for polyadic \emph{finite} sets and follows from an application of the M\"obius inversion formula, due to Gian-Carlo Rota. 

Recall that, given a finite\footnote{More generally, throughout this paragraph we could consider a \emph{locally finite} poset, i.e.\ a poset $P$ such that, for all $x,y\in P$, the interval $[x,y]\coloneqq\{z\in P\mid x\leq z\leq y\}$ is finite.} poset $P$, the \emph{incidence algebra} of $P$ has as elements the functions $f\colon P\times P\to \mathbb{R}$ satisfying $f(x,y)=0$ whenever ${x\not\leq y}$. These functions form an associative algebra over the reals with respect to pointwise sum, pointwise multiplication by scalars, and the convolution product $h=fg$ defined as
\[
h(x,y)\coloneqq\sum_{x\leq z\leq y} f(x,z)g(z,y).
\]
The identity element of the incidence algebra is the Kronecker function $\delta(x,y)$ satisfying $\delta(x,y)=1$ if $x=y$ and $\delta(x,y)=0$ otherwise.
The \emph{zeta function} of $P$ is the element $\zeta(x,y)$ of the incidence algebra such that $\zeta(x,y)=1$ if $x\leq y$ and $\zeta(x,y)=0$ otherwise. The function $\zeta$ is invertible in the incidence algebra, see \cite[Proposition~1]{Rota1964}, and its (two-sided) inverse $\mu$ is called the \emph{M\"obius function of $P$}.

\begin{lemma}[M\"obius inversion formula~\cite{Rota1964}]
Let $P$ be a finite poset and let $f_1,f_2\colon P\to \mathbb{R}$ be any two functions. For all $y\in P$,
\[
f_1(y)=\sum_{x\leq y}{f_2(x)} \enspace \Longleftrightarrow \enspace f_2(y)=\sum_{x\leq y}{f_1(x) \mu(x,y)},
\]
where $\mu$ is the M\"obius function of $P$.
\end{lemma}

\begin{proposition}\label{pr:pointwise-iso-kernel}
Suppose that $\A$ is $\Q$-well-founded, and let $F, G$ be polyadic finite sets on $\A$. If $F$ and $G$ are pointwise isomorphic, then so are their Stirling kernels $\sk{F}$ and $\sk{G}$.
\end{proposition}
\begin{proof}
Let $F,G\colon \A^\op\to\FinSet$ be any two polyadic finite sets, and suppose that they are pointwise isomorphic. We must prove that $\sk{F}(n)\cong \sk{G}(n)$ for all $n\in \A$. Let us fix an arbitrary object $n\in \A$. By Proposition~\ref{p:reconstruction}\ref{coprod-formula}, we know that 
\[
F(n)\cong \coprod_{n\epi m} \sk{F}(m).
\] 
As the left-hand side is a finite set, there exist finitely many quotients 
\[
q_1\colon n\epi m_1, \ \ldots, \ q_l\colon n\epi m_l
\] 
such that $F(n)\cong \coprod_{i=1}^l\sk{F}(m_i)$. Similarly, there are finitely many quotients 
\[
q_{l+1}\colon n\epi m_{l+1}, \ \ldots, \ q_u\colon n\epi m_{u}
\] 
such that $G(n)\cong \coprod_{i=l+1}^u\sk{G}(m_i)$. Let $P$ be the sub-poset of $\Q(n)$ defined by the elements $q_1,\ldots, q_u$ and the identity of $n$. For convenience of notation, we let $m_0\coloneqq n$ and let $q_0\colon n \epi m_0$ be the identity. Note that $q_0$ is the top element of $P$. Clearly, $F(n)\cong \coprod_{i=0}^u\sk{F}(m_i)$ and $G(n)\cong \coprod_{i=0}^u\sk{G}(m_i)$. 

Consider the functions on $P$ determined, for all $i\in \{0,\ldots, u\}$, by the cardinalities of $F(m_i)$ and $\sk{F}(m_i)$, respectively:
\[
f_1\colon P\to \N, \ \ f_1(q_i)\coloneqq |F(m_i)|
\]
and
\[
f_2\colon P\to \N, \ \ f_2(q_i)\coloneqq |\sk{F}(m_i)|.
\]
Then $f_1(q_0)=\sum_{x\in P} f_2(x)$ and so, by the M\"{o}bius inversion formula, 
\[
f_2(q_0)=\sum_{x\in P} f_1(x)\mu(x,q_0)
\] 
where $\mu$ is the M\"{o}bius function of the poset $P$. Reasoning in a similar manner for the functions $g_1, g_2\colon P\to \N$ defined by $g_1(q_i)\coloneqq |G(m_i)|$ and $g_2(q_i)\coloneqq |\sk{G}(m_i)|$, we get  
\[
g_2(q_0)=\sum_{x\in P} g_1(x)\mu(x,q_0).
\]
By assumption, for all $i\in\{0,\ldots, u\}$, we have $|F(m_i)|=|G(m_i)|$ and so $f_1(q_i)=g_1(q_i)$. We conclude that $f_2(q_0)=g_2(q_0)$, i.e.\ $\sk{F}(n)\cong \sk{G}(n)$.
\end{proof}

If the category $\A$ has pushouts, then Proposition~\ref{pr:pointwise-iso-kernel} can be proved by exploiting the inclusion-exclusion principle (see e.g.\ \cite[\S 2.1]{Stanley1}), rather than the M\"{o}bius inversion formula, in which case we can dispense with the assumption that $\A$ be $\Q$-well-founded. This is an immediate consequence of a result in~\cite{DJR2021}, and will come in handy in Section~\ref{s:beyond-loc-finite}.
\begin{proposition}\label{pr:pointwise-iso-kernel-pushouts}
Suppose that $\A$ has pushouts, and let $F, G$ be polyadic finite sets on $\A$. If $F$ and $G$ are pointwise isomorphic, then so are their Stirling kernels $\sk{F}$ and $\sk{G}$.
\end{proposition}
\begin{proof}
This was proved in \cite[Lemma~12]{DJR2021} whenever $F,G\colon \A^{\op}\to \FinSet$ are functors sending pushout squares in $\A$ consisting of quotients to pullback squares in $\FinSet$. (The proof of the aforementioned result relies on the inclusion-exclusion principle.) 

In turn, every polyadic (finite) set has this property. Just observe that, by Lemmas~\ref{l:amalg-quasi-pull} and~\ref{l:quot-inj-polyadic}, a polyadic set $\A^{\op}\to\Set$ sends pushout squares of quotients in $\A$ to quasi-pullbacks of injections. But a quasi-pullback in $\Set$ consisting of injections is a pullback, because the unique mediating map into the pullback is both injective and surjective.
\end{proof}

\section{Homomorphism Counting in Locally Finite Categories}\label{s:hom-count-finite}
We shall now exploit the framework of Section~\ref{s:polyadic-sets}, and in particular Proposition~\ref{pr:pointwise-iso-kernel}, to establish a homomorphism counting result for so-called locally finite categories (Theorem~\ref{thm:right-combinatorial-cats} below).

\begin{definition}
A category $\A$ is \emph{locally finite} if, for all objects $a,b\in\A$, the set $\hom_{\A}(a,b)$ is finite.
\end{definition}

To start with, we record for future reference a well known and easily proved fact about finite monoids. Recall that a monoid $(M,{\cdot},1)$ satisfies the \emph{left-cancellation law} provided that, for all $x,y,z\in M$,
\[
x\cdot y=x\cdot z \enspace \Longrightarrow \quad y=z.
\]

\begin{lemma}\label{l:left-can-group}
A finite monoid satisfying the left-cancellation law is a group.
\end{lemma}

\begin{theorem}\label{thm:right-combinatorial-cats}
Let $\A$ be a locally finite category admitting a proper factorisation system $(\Q,\M)$ such that $\A$ is $\Q$-well-founded. Then $\A$ is right-combinatorial.
\end{theorem}
\begin{proof}
Let $\A$ be as in the statement. For any object $a\in\A$, the representable functor $\yo_a\colon \A^\op\to\FinSet$ is a polyadic finite set by Lemma~\ref{l:repr-polyadic-set}, and its Stirling kernel sends an object $c\in \A$ to the finite set $\M(c,a)$ consisting of the embeddings $c\emb a$ (see Example~\ref{ex:kernel-repr}). 

Now, if $a,b\in \A$ are any two objects such that $\yo_a(c)\cong \yo_b(c)$ for all $c\in \A$ then, by Proposition~\ref{pr:pointwise-iso-kernel}, $\M(c,a)\cong \M(c,b)$ for all $c\in \A$. In particular, as $\M(a,a)$ is non-empty (it contains the identity arrow), there exists an embedding $i\in \M(a,b)$. Similarly, there exists an embedding $j\in\M(b,a)$. Note that, by Lemma~\ref{l:factorisation-properties}\ref{compositions}, $j\circ i\in\M(a,a)$. Lemma~\ref{l:left-can-group}, combined with the fact that every embedding is a monomorphism, entails that the set $\M(a,a)$ equipped with the composition operation is a group. So, $j\circ i$ has an inverse. It follows by Lemma~\ref{l:factorisation-properties}\ref{isos},\ref{cancellation-e} that $j$ is an isomorphism.
\end{proof}

\begin{remark}
A variant of Theorem~\ref{thm:right-combinatorial-cats}, where $\Q$ is the collection of extremal epimorphisms and $\M$ is the collection of monomorphisms, and each poset of embeddings $\M(a)$ is finite, was proved by Pultr \cite[Theorem~2.2]{pultr1973isomorphism} exploiting a direct generalisation of Lov\'{a}sz' original counting argument~\cite{Lovasz1967}.

Furthermore, reasoning along the same lines as in the previous proof, we get the following variant of Theorem~\ref{thm:right-combinatorial-cats} by applying Proposition~\ref{pr:pointwise-iso-kernel-pushouts} instead of Proposition~\ref{pr:pointwise-iso-kernel}: \emph{A locally finite category admitting pushouts and a proper factorisation system is right-combinatorial}. This result was first proved in \cite[Theorem~5]{DJR2021}.
\end{remark}

Theorem~\ref{thm:right-combinatorial-cats} admits a dual version (cf.\ Remark~\ref{rem:dual-proper-f-s} in the appendix):

\begin{theorem}\label{thm:left-combinatorial-cats}
Let $\A$ be a locally finite category admitting a proper factorisation system $(\Q,\M)$ such that $\A$ is $\M$-well-founded. Then $\A$ is left-combinatorial.
\end{theorem}

We conclude this section with several examples of applications of Theorems~\ref{thm:right-combinatorial-cats} and~\ref{thm:left-combinatorial-cats}.

\begin{example}
Let $\V$ be any variety of universal algebras (regarded as a category, with morphisms the homomorphisms). By Theorems~\ref{thm:right-combinatorial-cats} and~\ref{thm:left-combinatorial-cats}, the full subcategory $\V_{\fin}$ of $\V$ consisting of the finite members is combinatorial. Just observe that, if $(\Q,\M)$ is the proper factorisation system consisting of surjective and injective homomorphisms, respectively, then $\V_{\fin}$ is both $\Q$-well-founded and $\M$-well-founded. For instance, the following classes of algebras are combinatorial: finite Boolean algebras, finite monoids, finite groups, and finite Abelian groups.
\end{example}

\begin{example}
Generalising the previous example, let $\C$ be any class of universal algebras closed under taking subalgebras. Then the usual factorisation system in the category of all algebras for the given algebraic signature, given by surjective and injective homomorphisms, respectively, restricts to $\C$. Thus, $\C$ is combinatorial. A similar fact holds if $\C$ is closed under taking homomorphic images. For example, the following classes of algebras are closed under homomorphic images and therefore combinatorial: finite regular semigroups, finite inverse semigroups, and finite $p$-groups (cf.\ \cite[Lemma~2.4.4]{Howie1995}, \cite[Lemma~7.35]{CP1967}, and \cite[Exercise~6(c) p.~32]{Serre2002}, respectively). 
\end{example}

\begin{example}\label{ex:relational-structures}
Let $\sigma$ be a relational signature, i.e.\ a (possibly infinite) set of relation symbols of finite arity, and let $\R(\sigma)$ be the category of $\sigma$-structures with their homomorphisms. Then the full subcategory $\R_{\fin}(\sigma)$ of $\R(\sigma)$ defined by the \emph{finite} $\sigma$-structures is combinatorial; for a finite signature~$\sigma$, this is precisely Lov\'{a}sz' homomorphism counting theorem~\cite{Lovasz1967}.

First, recall that epimorphisms and monomorphisms in $\R(\sigma)$ coincide, respectively, with the surjective and injective homomorphisms. Further, strong epimorphisms (respectively, strong monomorphisms) coincide with the epimorphisms (respectively, monomorphisms) that reflect the relation symbols. The same holds in $\R_{\fin}(\sigma)$.

Now, note that the factorisation system $(\Q,\M)$ in $\R_{\fin}(\sigma)$, where $\Q$ consists of the strong epimorphisms and $\M$ of the monomorphisms, is proper and $\R_{\fin}(\sigma)$ is $\Q$-well-founded. Thus, $\R_{\fin}(\sigma)$ is right-combinatorial by Theorem~\ref{thm:right-combinatorial-cats}. On the other hand, the factorisation system $(\Q',\M')$ in $\R_{\fin}(\sigma)$, where $\Q'$ consists of the epimorphisms and $\M'$ of the strong monomorphisms, is also proper and $\R_{\fin}(\sigma)$ is $\M'$-well-founded. So, $\R_{\fin}(\sigma)$ is left-combinatorial in view of Theorem~\ref{thm:left-combinatorial-cats}. 

Therefore, $\R_{\fin}(\sigma)$ is combinatorial. Observe that, when $\sigma$ is infinite, the category $\R_{\fin}(\sigma)$ need not be $\Q'$-well-founded nor $\M$-well-founded.
\end{example}

\section{Beyond Locally Finite Categories}\label{s:beyond-loc-finite}
In Theorem~\ref{thm:right-combinatorial-cats} we saw that a large class of locally finite categories is right-combinatorial. The main result of this section (Theorem~\ref{th:Lovasz-finite-type} below) is a `local' extension of this fact to categories that need not be locally finite. This result is then specialised to the case of locally finitely presentable categories.

\subsection{Nerves and hom-spaces}\label{s:nerves-hom-spaces}
Throughout this section we fix a category $\D$ admitting a proper factorisation system $(\Q,\M)$, and a full subcategory $\C$ of $\D$ such that:
\begin{enumerate}[label=(\roman*)]
\item $\C$ is a \emph{dense} subcategory of $\D$, i.e.\ every $a\in \D$ is the colimit of the canonical diagram given by the forgetful functor $\C\down a \to \D$.
\item $\C$ has all finite colimits and they are preserved by the inclusion functor $\C\hookrightarrow \D$.
\item For every composite $a \epi b \emb c$ in $\D$, if $a,c\in \C$ then also $b\in \C$. 
\end{enumerate}

\begin{remark}\label{rem:assumptions-dense-subcat}
Because $\C$ is closed in $\D$ under finite colimits by item~(ii), the diagrams of the form $\C\down a \to \D$ in item~(i) are automatically directed. %Just observe that any two morphisms $c_1\to d$ and $c_2\to d$, with $c_1,c_2\in \C$, induce a morphism $c_1+c_2\to d$ such that $c_1+c_2\in \C$. 

Further, item~(iii) amounts to saying that the proper factorisation system $(\Q,\M)$ in $\D$ restricts to a (proper) factorisation system in $\C$. This condition is satisfied, e.g., if $\C$ is closed in $\D$ under $\Q$-images (i.e., given a quotient $a\epi b$ in $\D$, if $a\in \C$ then also $b\in\C$) or under $\M$-subobjects. 
Also, note that item~(iii) implies that $\C$ is closed in $\D$ under isomorphisms.
\end{remark}

Given an object $a\in\D$, let $N_a\colon \C^\op\to\Set$ be the restriction of the presheaf $\yo_a\colon \D^\op\to\Set$ to $\C^\op$.
We consider the \emph{nerve functor}
\[
N\colon \D\to \w{\C}, \quad a\mapsto N_a.
\]
In Theorem~\ref{th:Lovasz-finite-type}, we will identify a class of objects of $\D$ such that any two objects in this class are isomorphic precisely when their images under the nerve functor are pointwise isomorphic.

It is useful to observe that the nerve functor is full and faithful because $\C$ is a dense subcategory of $\D$ (see e.g.\ \cite[p.~218]{AHS1990}). Hence, $N$ is conservative (that is, it reflects isomorphisms). Explicitly, this means that a morphism $f\colon a\to b$ in $\D$ is an isomorphism if, and only if, the function
\[
f\circ -\colon N_a(c)\to N_b(c)
\]
is a bijection for every $c\in \C$.

Since we want to be able to count morphisms from objects of $\C$, it makes sense to restrict our attention to those objects of $\D$ that look `finite' from the viewpoint of every object of $\C$:
\begin{definition}
An object $a\in\D$ is of \emph{finite $\C$-type} provided that the set $\hom_{\D}(c,a)$ is finite for each $c\in\C$.
\end{definition}

\begin{example}
The previous definition makes sense for all categories $\D$ equipped with a (full) subcategory $\C$. For instance, let $\D$ be the opposite of the category of groups and group homomorphisms, and let $\C$ be the full subcategory of $\D$ defined by the finite groups. Then any finitely generated group is of finite $\C$-type, because there are finitely many homomorphisms from a finitely generated group to a finite one.
\end{example}

\begin{lemma}\label{lem:nerve-discrete-polyadic-space}
Let $a\in\D$ be an object of finite $\C$-type. The following statements hold:
\begin{enumerate}[label=(\alph*)]
\item\label{nerv-pol-set} The nerve $N_a\colon \C^\op\to\FinSet$ is a polyadic finite set.
\item\label{Stirling-nerve} The Stirling kernel of $N_a$ sends an object $c\in\C$ to the set $\M(c,a)$ of embeddings $c\emb a$.
\end{enumerate} 
\end{lemma}
\begin{proof}
First we show that, for any polyadic set $F\colon \D^\op\to\Set$, its restriction to $\C^\op$ is also a polyadic set. Item (a) then follows at once from Lemma~\ref{l:repr-polyadic-set}.

Let $F\colon \D^\op\to\Set$ be an arbitrary polyadic set. Since $\C$ has pushouts, by Lemma~\ref{l:amalg-quasi-pull} it suffices to show that $F$ sends pushout squares in $\C$ to quasi-pullbacks in $\Set$. Consider a pushout square in $\C$ as on the left-hand side below and the corresponding square in $\Set$ as on the right-hand side.
\begin{equation*}
\begin{tikzcd}
{\cdot} \arrow{d}[swap]{g} \arrow{r}{f} & a \arrow{d} \\
b \arrow{r} & c
\end{tikzcd}
\ \ \ \ \ \ \ \ \ \ \ \ \ 
\begin{tikzcd}
F(\cdot) & F(a) \arrow{l}[swap]{F(f)} \\
F(b) \arrow{u}{F(g)} & F(c) \arrow{u} \arrow{l}
\end{tikzcd}
\end{equation*}
If $x\in F(a)$ and $y\in F(b)$ are such that $F(f)(x)=F(g)(y)$, by the amalgamation property of $F$ there exist $d\in \D$, morphisms $f'\colon b\to d$ and $g'\colon a\to d$ such that $f'\circ g=g'\circ f$, and an element $z\in F(d)$ such that $F(g')(z)=x$ and $F(f')(z)=y$. Because the inclusion $\C\hookrightarrow \D$ preserves pushouts, there exists a unique morphism $h\colon c\to d$ making the ensuing diagram commute. The element $F(h)(z)\in F(c)$ then witnesses the fact that the rightmost square above is a quasi-pullback.

Item (b) follows by reasoning as in Example~\ref{ex:kernel-repr}, using the fact that the factorisation system in $\D$ restricts to a proper factorisation system in $\C$.
\end{proof}

Consider any two objects $a,b\in\D$ and the canonical colimit cocone 
\[
\{b_i\to b\mid i\in I\}
\] 
in $\D$ where each $b_i$ belongs to $\C$. We have an isomorphism
\[
\yo_a(b)\cong \lim_{i\in I}\yo_a(b_i)
\]
in the category of sets. Let us assume that $a$ is of finite $\C$-type, so that each $\yo_a(b_i)=N_a(b_i)$ is finite. Note that the diagram consisting of the $b_i$'s is directed (see Remark~\ref{rem:assumptions-dense-subcat}). So, if we equip the sets $\yo_a(b_i)$ with the discrete topologies, the hom-set $\yo_a(b)$ carries a natural Stone (i.e.\ compact, Hausdorff and zero-dimensional) topology, namely the inverse limit topology. Explicitly, this is the topology generated by the sets of the form
\begin{equation*}
\begin{tikzcd}
\mathscr{U}_{\langle u,v\rangle}\coloneqq \{h\in \yo_a(b)\mid h\circ u=v\}
\end{tikzcd}
\quad \quad \quad \quad 
\begin{tikzcd}[column sep=1em]
{} & c \arrow{dl}[swap]{u} \arrow{dr}{v} & {} \\
b \arrow{rr}{h} & {} & a 
\end{tikzcd}
\end{equation*}
for $c\in \C$, $u\in N_b(c)$ and $v\in N_a(c)$.

Let us denote by $E_a(b)$ the `hom-space' obtained by endowing the set $\yo_a(b)$ with the Stone topology just described.\footnote{Although we shall not need this fact in the following, let us point out that the assignment $b\mapsto E_a(b)$ yields a functor $E_a\colon \D^\op\to\Stone$ which extends $\yo_a\colon \D^\op\to\Set$.}
Next, we prove some useful properties of the space of endomorphisms $E_a(a)$:

\begin{lemma}\label{lem:finite-type-topology}
The following statements hold for every $a\in\D$ of finite $\C$-type:
\begin{enumerate}[label=(\alph*)]
\item $E_a(a)$ is a topological monoid with respect to composition.
\item $\M(a,a)$ is a closed submonoid of $E_a(a)$.
\end{enumerate}
\end{lemma}
\begin{proof}
For item (a), we must prove that the composition operation 
\[
\circ \colon E_a(a)\times E_a(a)\to E_a(a), \enspace (f,g)\mapsto f\circ g
\] 
is continuous. Consider $(f,g)\in E_a(a)\times E_a(a)$ and an open neighbourhood $\mathscr{U}_{\langle u,v\rangle}$ of $f\circ g$. Then the set $\mathscr{U}_{\langle g\circ u,v\rangle}\times \mathscr{U}_{\langle u,g\circ u \rangle}$ is an open neighbourhood of $(f,g)$ whose image is contained in $\mathscr{U}_{\langle u,v\rangle}$. Just observe that, for all $(f',g')\in \mathscr{U}_{\langle g\circ u,v\rangle}\times \mathscr{U}_{\langle u,g\circ u \rangle}$, we have $f'\circ g'\circ u=f'\circ g\circ u=v$. Hence, the composition operation is continuous.

For item (b), suppose that $f\in E_a(a)$ is not an embedding. We must find an open neighbourhood $V$ of $f$ disjoint from $\M(a,a)$. Because $f$ is not an isomorphism and the nerve functor $N\colon \D\to\w{\C}$ is conservative, there exists $c\in\C$ such that the map 
\[
f\circ -\colon N_a(c)\to N_a(c)
\] 
is not a bijection. Since $N_a(c)$ is a finite set, $f\circ -$ cannot be injective. Hence, there exist distinct morphisms $u,v\in N_a(c)$ such that $f\circ u=f\circ v$. The set $V\coloneqq \mathscr{U}_{\langle u,f\circ u\rangle}\cap \mathscr{U}_{\langle v,f\circ v\rangle}$ is clearly an open neighbourhood of $f$. We claim that $V$ is disjoint from $\M(a,a)$. Assume, by way of contradiction, that $g\colon a\emb a$ is an embedding that belongs to $V$. Then $g\circ u=f\circ u=f\circ v=g\circ v$. As $g$ is a monomorphism we get $u=v$, a contradiction.
\end{proof}

\begin{remark}
Recall that a topological monoid is profinite if, and only if, its underlying space is Stone~\cite{Numakura1957}. Thus, it follows by Lemma~\ref{lem:finite-type-topology} that $E_a(a)$ and $\M(a,a)$ are profinite monoids.
\end{remark}

Lemma~\ref{l:left-can-group} for finite monoids admits a non-trivial generalisation to topological monoids, which we now recall. This result will be applied in (the proof of) Theorem~\ref{th:Lovasz-finite-type} to the topological monoids $\M(a,a)$ as in Lemma~\ref{lem:finite-type-topology}.
\begin{lemma}[{Numakura's Lemma~\cite[Lemma 2L]{Numakura1952}}]\label{lem:numakura-lemma} 
Let $S$ be a topological monoid whose topology is compact and Hausdorff. If $S$ satisfies the left-cancellation law then it is a group.
\end{lemma}

We are interested in determining when two non-isomorphic objects $a,b\in\D$ of finite $\C$-type can be distinguished by counting the number of morphisms from objects of $\C$. To this end, we will assume that $a$ and $b$ satisfy an additional `separability' property. (The latter is satisfied by any object of a locally finitely presentable category $\D$, for an appropriate choice of the subcategory $\C$, cf.~the proof of Theorem~\ref{th:Lovasz-lfp-categories}.)

\begin{definition}\label{def:C-separable}
An object of $\D$ is \emph{$\C$-separable} if it is the colimit of a directed diagram $J$ in $\C$ such that:
\begin{enumerate}[label=(\roman*)]
\item The colimit cocone consists of embeddings.
\item For every compatible cocone of embeddings over $J$, the unique connecting morphism is also an embedding.
\end{enumerate}
\end{definition}

Intuitively, an object $a$ is $\C$-separable if there exists a (directed) family $S$ of $\M$-subobjects of $a$ such that, for all objects $d\in\D$, $a$ can be embedded into $d$ precisely when all objects in $S$ can be embedded coherently into $d$ (and, moreover, $a$ is the union of its subobjects in $S$).  

\begin{remark}\label{rem:small-diag-embeddings}
Note that, by Lemma~\ref{l:factorisation-properties}\ref{cancellation-m}, item~(i) in the previous definition implies that the diagram $J$ consists entirely of embeddings.
\end{remark}

We are now in a position to prove the main result of this section:

\begin{theorem}\label{th:Lovasz-finite-type}
For any two $\C$-separable objects $a,b\in\D$ of finite $\C$-type, 
\[
a\cong b \enspace \Longleftrightarrow \enspace \hom_{\D}(c,a)\cong \hom_{\D}(c,b) \ \text{ for all } c\in \C.
\]
\end{theorem}
\begin{proof}
Let $a,b\in \D$ be as in the statement. For the non-trivial direction, suppose that $\hom_{\D}(c,a)\cong \hom_{\D}(c,b)$ for all $c\in\C$. By assumption, $b$ is the colimit of a directed diagram $\{b_i\mid i\in I\}$ in $\C$ satisfying the conditions in Definition~\ref{def:C-separable}.
Therefore, $b\cong \colim_{D} b_i$ entails
\begin{equation*}
\yo_a(b)\cong \lim_{i\in I} N_a(b_i),
\end{equation*}
i.e.\ the map associating with a compatible cocone $\{b_i\to a\mid i\in I\}$ the unique connecting morphism $b\to a$ is a bijection. If each finite set $N_a(b_i)$ is equipped with the discrete topology, then the induced inverse limit topology $\tau$ on $\yo_a(b)$ coincides with the topology of $E_a(b)$ and so we have a homeomorphism of Stone spaces 
\begin{equation}\label{eq:homeo-E_a(b)-inv-lim}
E_a(b)\cong \lim_{i\in I} N_a(b_i).
\end{equation}
Just observe that the diagram of the $b_i$'s is a subdiagram of the canonical diagram given by (the image of) the forgetful functor $\C\down b \to \D$, and so the identity map $E_a(b)\to (\yo_a(b),\tau)$ is continuous. In turn, since any two comparable compact Hausdorff topologies on a set must coincide, the two topologies are one and the same.

The space $\lim_{i\in I} \M(b_i,a)$, endowed with the inverse limit topology, can be identified with a subspace of $\lim_{i\in I} N_a(b_i)$. We claim that the homeomorphism in~\eqref{eq:homeo-E_a(b)-inv-lim} restricts to a homeomorphism between $\M(a,b)$ and $\lim_{i\in I} \M(b_i,a)$.
If a compatible cocone $\{b_i\to a\mid i\in I\}$ induces a connecting morphism $b\to a$ that is an embedding then Lemma~\ref{l:factorisation-properties}\ref{compositions}, combined with item~(i) in the definition of a $\C$-separable object and Remark~\ref{rem:small-diag-embeddings}, entails that each $b_i\to a$ is an embedding. Conversely, if $\{b_i\to a\mid i\in I\}$ is a compatible cocone consisting of embeddings then the unique connecting morphism is an embedding by item~(ii) in the definition of a $\C$-separable object. Hence, $\M(b,a)$ is homeomorphic to $\lim_{i\in I} \M(b_i,a)$. 

By assumption, the nerves $N_a, N_b\colon \C^\op\to\FinSet$ are pointwise isomorphic. It follows by Proposition~\ref{pr:pointwise-iso-kernel-pushouts} and Lemma~\ref{lem:nerve-discrete-polyadic-space}\ref{nerv-pol-set} that their Stirling kernels are also pointwise isomorphic. So, by Lemma~\ref{lem:nerve-discrete-polyadic-space}\ref{Stirling-nerve},  
\[
\M(b_i,a)\cong \M(b_i,b)
\] 
for every $i\in I$. As the sets $\M(b_i,b)$ are non-empty (they contain the colimit maps), we conclude that the space $\M(b,a)$ is the inverse limit of non-empty finite discrete spaces and thus $\M(b,a)\neq\emptyset$ (see e.g.\ \cite[Theorem~2-85]{HY1988}). That is, there exists an embedding $\alpha\colon b\emb a$. By symmetry, there exists also an embedding $\beta\colon a\emb b$.

The composite $\alpha\circ \beta$ belongs to the monoid $\M(a,a)$, which is a Stone topological monoid by Lemma~\ref{lem:finite-type-topology}. It follows by Numakura's Lemma that $\alpha\circ\beta$ has an inverse  and so $\alpha$ is an isomorphism by Lemma~\ref{l:factorisation-properties}\ref{isos},\ref{cancellation-e}.
\end{proof}

\subsection{Locally finitely presentable categories}\label{s:lfp}
In this section we specialise Theorem~\ref{th:Lovasz-finite-type} to the case of locally finitely presentable categories. To start with, we recall some basic definitions; for a more thorough treatment, the reader can consult e.g.~\cite{AR1994}. 

An object $a$ of a category $\A$ is \emph{finitely presentable} (respectively, \emph{finitely generated}) if the associated covariant hom-functor 
\[
\hom_{\A}(a,-)\colon \A\to\Set
\]
preserves directed colimits (respectively, directed colimits of monomorphisms). A category $\A$ is said to be \emph{locally finitely presentable} if it is cocomplete, every object is a directed colimit of finitely presentable objects, and there exists, up to isomorphism, only a set of finitely presentable objects.

Let $\A$ be a locally finitely presentable category. Then $\A$ admits a proper factorisation system $(\Q,\M)$ where $\Q$ consists of the strong epimorphisms and $\M$ of the monomorphisms. Further, an object $a\in \A$ is finitely generated if, and only if, there exist a finitely presentable object $b\in \A$ and a quotient (i.e., a strong epimorphism) $b\epi a$. For a proof of these facts see, e.g., \cite[Propositions~1.61 and~1.69(ii)]{AR1994}. Throughout, we denote by $\A_{\fp}$ and $\A_{\fg}$ the full subcategories of $\A$ consisting, respectively, of the finitely presentable and finitely generated objects.

In this context, the role of the dense subcategory $\C$ in Theorem~\ref{th:Lovasz-finite-type} is played by $\A_{\fg}$. However, it is an easy observation that any object of finite $\A_{\fp}$-type is also of finite $\A_{\fg}$-type:
\begin{lemma}\label{l:finite-fp-fg-type}
Let $\A$ be a locally finitely presentable category. An object of $\A$ is of finite $\A_{\fp}$-type if, and only if, it is of finite $\A_{\fg}$-type.
\end{lemma}
\begin{proof}
For the non-trivial direction, suppose that $a$ is of finite $\A_{\fp}$-type and consider an arbitrary $b\in \A_{\fg}$. If $f\colon c\epi b$ is a quotient with $c$ finitely presentable, the map $-\circ f \colon \hom_{\A}(b,a)\to \hom_{\A}(c,a)$ is injective. Since the latter set is finite, so is the former.
\end{proof}

\begin{theorem}\label{th:Lovasz-lfp-categories}
Let $\A$ be a locally finitely presentable category. For any two objects $a,b\in\A$ of finite $\A_{\fp}$-type, 
\[
a\cong b \enspace \Longleftrightarrow \enspace \hom_{\A}(c,a)\cong \hom_{\A}(c,b) \ \text{ for all } c\in \A_{\fg}.
\]
\end{theorem}
\begin{proof}
Recall that $\A_{\fg}$ is a dense subcategory of $\A$ because so is $\A_{\fp}$ (see e.g.\ \cite[Proposition~1.22]{AR1994}), and it is closed in $\A$ under $\Q$-images (see e.g.\ \cite[Proposition~1.69(i)]{AR1994}). Further, it is a folklore result that $\A_{\fg}$ has all finite colimits and these are preserved by the inclusion functor $\A_{\fg}\hookrightarrow \A$.
Finally, every object of $\A$ is $\A_{\fg}$-separable (cf.\ \cite[Proposition~1.62 and Theorem~1.70]{AR1994}). Therefore, an application of Theorem~\ref{th:Lovasz-finite-type} with $\D\coloneqq \A$ and $\C\coloneqq \A_{\fg}$, combined with Lemma~\ref{l:finite-fp-fg-type}, yields the statement.
\end{proof}

We now specialise Theorem~\ref{th:Lovasz-lfp-categories} to ind- and pro-categories. The ensuing results are then applied in concrete cases in Section~\ref{s:examples}. A further application of Theorem~\ref{th:Lovasz-lfp-categories}, this time to categories of coalgebras for certain comonads, is presented in Section~\ref{s:coalg-finite-rank}.

Given a (essentially small) category $\C$, denote its \emph{ind-completion} and \emph{pro-completion} by $\ind(\C)$ and $\pro(\C)$, respectively (see e.g.~\cite[\S VI.1]{Johnstone1986}). These are, respectively, the free cocompletion of $\C$ under filtered colimits, and the free completion of $\C$ under cofiltered limits. Up to an equivalence of categories, we can and will identify $\C$ with a full subcategory of $\ind(\C)$ and $\pro(\C)$, respectively.

\begin{corollary}\label{cor:counting-indC}
Let $\C$ be an essentially small category with finite colimits that is closed under strong epimorphic images in $\ind(\C)$. For all objects $a,b\in\ind(\C)$ of finite $\C$-type, 
\[
a\cong b \enspace \Longleftrightarrow \enspace \hom_{\ind(\C)}(c,a)\cong \hom_{\ind(\C)}(c,b) \ \text{ for all } c\in \C.
\]
\end{corollary}

\begin{proof}
If $\C$ is essentially small and has finite colimits then $\ind(\C)$ is a locally finitely presentable category; see e.g.\ \cite[Corollary~VI.1.3]{Johnstone1986} and \cite[Theorem~1.46]{AR1994}. 
Furthermore, every finitely presentable object of $\ind(\C)$ is isomorphic to an object of $\C$ (this follows from the fact that $\C$ has finite colimits, hence it is idempotent-complete, combined with \cite[Exercise~6.1(iii)]{KS2006}). The same is true of finitely generated objects of $\ind(\C)$, as $\C$ is closed under strong epimorphic images in $\ind(\C)$. Therefore, the statement follows directly from Theorem~\ref{th:Lovasz-lfp-categories}.
\end{proof}

We record for future reference the dual version of Corollary~\ref{cor:counting-indC}. 
\begin{corollary}\label{cor:counting-proC}
Let $\C$ be an essentially small category with finite limits that is closed under strong subobjects in $\pro(\C)$. Let $a,b\in\pro(\C)$ be such that the sets $\hom_{\pro(\C)}(a,c)$ and $\hom_{\pro(\C)}(b,c)$ are finite for all $c\in\C$. Then
\[
a\cong b \enspace \Longleftrightarrow \enspace \hom_{\pro(\C)}(a,c)\cong \hom_{\pro(\C)}(b,c) \ \text{ for all } c\in \C.
\]
\end{corollary}

%%%%%%%%%%%%%%%%%%%%%%%%%%%%%%
\subsection{Coalgebras for comonads of finite rank}\label{s:coalg-finite-rank}
In this section we specialise Theorem~\ref{th:Lovasz-lfp-categories} to categories of coalgebras for comonads on locally finitely presentable categories. 

Whenever $T$ is a comonad on a category $\A$, we write $\EM(T)$ for the category of Eilenberg-Moore coalgebras for $T$, and 
\[\begin{tikzcd}[column sep=1.5em]
\A \arrow[yshift=-6pt]{rr}[swap]{F} & {\text{\scriptsize{$\bot$}}} & \EM(T) \arrow[yshift=6pt]{ll}[swap]{U}
\end{tikzcd}\]
for the associated adjunction. For the latter notions, see Appendix~\ref{s:finite-rank}.
In the next result, we will assume that $\A$ is locally finitely presentable and $T$ is of \emph{finite rank}, and so the category $\EM(T)$ is also locally finitely presentable. Comonads of finite rank were defined by Diers in~\cite{Diers1986}; for a definition and some basic facts, we refer the reader to Appendix~\ref{s:finite-rank}. 

\begin{corollary}\label{cor:hom-counting-comonads-finite-rank}
Let $\A$ be a locally finitely presentable category and let $T$ be a comonad of finite rank on $\A$ with associated adjunction $U\dashv F$. The following statements are equivalent for all objects $a,b\in\A$ of finite $\A_{\fp}$-type: 
\begin{enumerate}
\item $F(a)\cong F(b)$.
\item For all $x\in \EM(T)_{\fg}$, 
\[
\hom_{\A}(U(x),a)\cong \hom_{\A}(U(x),b).
\]
\end{enumerate}
\end{corollary}
\begin{proof}
Since $T$ is of finite rank, Theorem~\ref{t:EM-lfp} in the appendix entails that the category $\EM(T)$ is locally finitely presentable and the forgetful functor 
\[
U\colon \EM(T)\to \A
\] 
preserves (and reflects) finitely presentable objects. 

Note that the object $F(a)$ is of finite $\EM(T)_{\fp}$-type whenever $a$ is of finite $\A_{\fp}$-type. Just observe that, for all $x\in \EM(T)_{\fp}$, 
\[
\hom_{\EM(T)}(x,F(a))\cong \hom_{\A}(U(x),a)
\]
which is a finite set because $U(x)$ is finitely presentable.
Therefore, the statement follows by an application of Theorem~\ref{th:Lovasz-lfp-categories}.
\end{proof}

\begin{remark}\label{rem:replace-fg-with-fp}
Suppose we are in the situation of the previous corollary. In view of Corollary~\ref{cor:fp-equal-fg} in the appendix, if the finitely presentable objects in $\A$ coincide with the finitely generated ones, then the same holds in $\EM(T)$. 

In this case, Corollary~\ref{cor:hom-counting-comonads-finite-rank} states that, for all objects $a,b\in\A$ of finite $\A_{\fp}$-type, $F(a)\cong F(b)$ if, and only if,
\[
\hom_{\A}(U(x),a)\cong \hom_{\A}(U(x),b)
\]
for all $x\in\EM(T)$ such that $U(x)$ is finitely presentable. 

This occurs, for instance, when $\A=\Set$ or $\A=\R(\sigma)$ is the category of $\sigma$-structures for a finite relational signature $\sigma$. A similar remark applies to Corollary~\ref{cor:hom-counting-comonads-finite-rank-relative} below.
\end{remark}

We include a `relative' version of Corollary~\ref{cor:hom-counting-comonads-finite-rank} which will be needed in Section~\ref{s:finite-variable-logic} for applications to finite-variable logics. To this end, given a functor $G\colon \C\to \D$ and a full subcategory $\tilde{\C}$ of~$\C$, we write $G[\tilde{\C}]$ for the full subcategory of $\D$ defined by the objects of the form $G(c)$ with $c\in\tilde{\C}$.
\begin{corollary}\label{cor:hom-counting-comonads-finite-rank-relative}
Let $\A'$ be a locally finitely presentable category and assume that there is a full and faithful functor $J\colon \A\hookrightarrow \A'$ with a left adjoint $H$. Let $T,T'$ be comonads on $\A$ and $\A'$, respectively, with associated adjunctions ${U\dashv F}$ and ${U'\dashv F'}$. 
Suppose $T'$ is of finite rank and the adjunction $H\dashv J$ restricts to functors $U[\EM_\fg(T)]\leftrightarrows U'[\EM_\fg(T')]$:
\[\begin{tikzcd}[column sep=1em]
\A \arrow[yshift=-6pt]{rr}[swap]{J} & {\text{\scriptsize{$\bot$}}} & \A' \arrow[yshift=6pt]{ll}[swap]{H} \\ 
U[\EM_\fg(T)] \arrow[hookrightarrow]{u} \arrow[dashed,yshift=-4pt]{rr} & {} & U'[\EM_\fg(T')] \arrow[hookrightarrow]{u} \arrow[dashed,yshift=4pt]{ll}
\end{tikzcd}\]
 The following are equivalent for all objects $a,b\in\A$ of finite $\A_{\fp}$-type: 
\begin{enumerate}
\item $F'J(a)\cong F'J(b)$.
\item For all $x\in \EM(T)_{\fg}$, 
\[
\hom_{\A}(U(x),a)\cong \hom_{\A}(U(x),b).
\]
\end{enumerate}
\end{corollary}
\begin{proof}
In view of Corollary~\ref{cor:hom-counting-comonads-finite-rank}, item~1 in the statement is equivalent to saying that, for all $x'\in\EM(T')_{\fg}$, 
\begin{equation}\label{eq:bij-T'-coalg}
\hom_{\A'}(U'(x),J(a))\cong \hom_{\A'}(U'(x),J(b)).
\end{equation}
In turn, it is an easy observation that the latter condition is equivalent to item 2 in the statement. This is essentially the content of \cite[Lemma~27]{DJR2021}; for the sake of completeness, we provide a proof. Suppose that equation~\eqref{eq:bij-T'-coalg} holds for all $x'\in\EM(T')_{\fg}$, and fix an arbitrary $x\in \EM(T)_{\fg}$. We have
\begin{align*}
\hom_{\A}(U(x),a) &\cong \hom_{\A'}(JU(x),J(a)) \tag{$J$ full and faithful}  \\
& \cong \hom_{\A'}(JU(x),J(b)) \tag{$JU(x)\in U'[\EM_\fg(T')]$} \\
&\cong \hom_{\A}(U(x),b). \tag{$J$ full and faithful}
\end{align*}
Conversely, suppose item 2 in the statement holds. For all $x'\in\EM(T')_{\fg}$,
\begin{align*}
\hom_{\A'}(U'(x),J(a)) &\cong \hom_{\A}(HU'(x),a) \tag{$H\dashv J$}  \\
& \cong \hom_{\A}(HU'(x),b) \tag{$HU'(x)\in U[\EM_\fg(T)]$} \\
&\cong \hom_{\A'}(U'(x),J(b)) \tag{$H\dashv J$}
\end{align*}
and so equation~\eqref{eq:bij-T'-coalg} holds. This concludes the proof.
\end{proof}

%%%%%%%%%%%%%%%%%%%%%%%%%%%%%%%%%%%%%%%%%%%%
\section{Examples}\label{s:examples}

\subsection{Trees} 
If $(P, {\leq})$ is a poset, then $C \subseteq P$ is a \emph{chain}  if it is linearly ordered. A \emph{forest} is a poset $(P,\leq)$ such that, for all $u\in P$, the set 
\[
\down u\coloneqq \{v\in P\mid v\leq u\}
\] 
is a finite chain. 
The \emph{covering relation} $\cvr$ associated with a partial order $\leq$ is defined by $u\cvr v$ if and only if $u<v$ and there is no $w$ such that $u<w< v$.
The \emph{roots} of a forest are the minimal elements. A \emph{tree} is a forest with at most one root (note that a tree is either empty, or has a unique root, the least element in the order).
Morphisms of trees are maps which preserve the root and the covering relation. 
The category of trees is denoted by $\T$. Monomorphisms and strong epimorphisms in $\T$ coincide, respectively, with the injective and surjective tree morphisms.

It is well known that $\T$ is a locally finitely presentable category in which the finitely presentable objects, which coincide with the finitely generated ones, are precisely the finite trees. Moreover, it is not difficult to see that a tree $(P,\leq)$ has finite $\T_{\fp}$-type if, and only if, it is \emph{finitely branching}. That is, for every $u\in P$, the set $\{v\in P\mid u\cvr v\}$ is finite.

Thus, Theorem~\ref{th:Lovasz-lfp-categories} entails at once the following result:
\begin{theorem}
Let $P,Q$ be any two finitely branching trees. Then $P\cong Q$ if and only if, for all finite trees $R$, the number of tree morphisms $R\to P$ is the same as the number of tree morphisms $R\to Q$.
\end{theorem}

\begin{remark}
In the last part of the proof of Theorem~\ref{th:Lovasz-finite-type}, we used the fact that 
\[
\forall i\in I. \ \M(b_i, a)\neq \emptyset \enspace \Longrightarrow \enspace \M(\colim_{i\in I} b_i, a)\neq\emptyset
\]
where, using the notation of the aforementioned theorem, the objects $b_i$ sit in the category $\C$ and $a$ is a $\C$-separable object of finite $\C$-type. 

This can be regarded as a generalisation of K\"{o}nig's Lemma for trees, stating that every finitely branching infinite tree contains an infinite simple path. Just observe that a countably infinite simple path $P_\omega$ is a colimit in $\T$ of finite simple paths $P_n$ of (increasing) length $n$. If $Q$ is an infinite tree, then for all $n$ there exists an embedding $P_n \emb Q$. If, in addition, $Q$ is finitely branching, then it has finite $\T_{\fp}$-type (and is $\T_{\fp}$-separable) and so there is an embedding $P_\omega\emb Q$.\footnote{In the specific case of trees, embeddings could be replaced by arbitrary arrows. Just note that every morphism in $\T$ whose domain is linearly ordered is automatically injective.}
In this sense, a version of K\"{o}nig's Lemma holds, in particular, in every locally finitely presentable category.
\end{remark}

\subsection{Profinite algebras}
In this section we focus on profinite universal algebras; a nice expository paper on the subject is~\cite{Banaschewski1972}.

Let $\V$ be an arbitrary variety of universal algebras, regarded as a category with morphisms the homomorphisms. The full subcategory $\V_{\fin}$ of $\V$ defined by the finite algebras is essentially small and has finite limits, which are computed in the category of sets. The same holds for any full subcategory $\C$ of $\V_{\fin}$ that is closed under subalgebras and finite products.

Let us fix a full subcategory $\C$ of $\V_{\fin}$ closed under subalgebras and finite products. The category $\pro(\C)$ can be identified with a full subcategory of the category $\K_{\V}$, whose objects are the topological $\V$-algebras carrying a compact Hausdorff topology and whose morphisms are the continuous homomorphisms.\footnote{Under this identification, an algebra in $\C$ is regarded as a topological algebra with respect to the discrete topology.} Explicitly, a topological algebra $A\in \K_{\V}$ belongs to $\pro(\C)$ if and only if, whenever $f, g\colon B\to A$ are distinct morphisms in $\K_{\V}$, there exists a morphism $h\colon A\to C$ with $C\in \C$ such that $h\circ f\neq h\circ g$. We shall refer to the objects of $\pro(\C)$ as \emph{pro-$\C$} algebras. If $\C=\V_{\fin}$, these coincide with the usual profinite $\V$-algebras.

The monomorphisms in $\K_{\V}$ are precisely the injective maps, i.e.\ the maps that provide (topological and algebraic) isomorphisms with the image. This holds in $\pro(\C)$ as well, because the latter is a reflective subcategory of $\K_{\V}$. In particular, the subobjects in $\pro(\C)$ can be identified with the closed subalgebras. For the previous assertions, cf.~\cite{Banaschewski1972} and the references therein. It follows that $\C$ is closed under (strong) subobjects in $\pro(\C)$. 

Finally, recall that a universal algebra $A$ is said to be \emph{finitely generated} if there exists a finite subset $S\subseteq A$ such that the inclusion-smallest subalgebra $\langle S\rangle$ of $A$ containing $S$ is $A$ itself. In the setting of topological algebras it is customary to relax the previous condition and say that a topological algebra $A$ is \emph{topologically finitely generated} if there exists a finite subset $S\subseteq A$ such that $\langle S\rangle$ is dense in the topology of $A$. 
Now, if $A\in \pro(\C)$ is topologically finitely generated and $C\in \C$, the set $\hom_{\pro(\C)}(A, C)$ is finite. Just observe that, if $S\subseteq A$ is a finite subset such that $\langle S\rangle$ is dense in $A$, the obvious restriction function
\[
\hom_{\pro(\C)}(A, C)\to C^S
\]
is injective because a continuous map into a Hausdorff space is completely determined by its behaviour on any dense subset of its domain.

Hence, the following is an immediate consequence of Corollary~\ref{cor:counting-proC}:
\begin{theorem}\label{t:counting-pro-C}
Let $\V$ be a variety of universal algebras, $\C\subseteq \V_{\fin}$ a class of finite algebras closed under subalgebras and finite products, and $A,B$ two topologically finitely generated pro-$\C$ algebras. Then $A\cong B$ if and only if, for all (discrete) algebras $C\in \C$, the number of continuous homomorphisms $A\to C$ is the same as the number of continuous homomorphisms $B\to C$. 
\end{theorem}

The previous result applies, e.g., when $\C$ is the class of finite lattices, or finite Heyting algebras, or finite semigroups, and provides a characterisation of the isomorphism relation for (topologically finitely generated) profinite lattices, profinite Heyting algebras, and profinite semigroups, respectively. 

\vspace{0.5em}
Next, we restrict our attention to topologically finitely generated profinite groups, whose topological structure is determined by the algebraic one.

Instantiating Theorem~\ref{t:counting-pro-C} with $\C$ the class of finite groups, we obtain the following result: \emph{Two topologically finitely generated profinite groups $G,H$ are isomorphic (as topological groups) if and only if, for all finite discrete groups $K$, the number of continuous group homomorphisms $G\to K$ is the same as the number of continuous group homomorphisms $H\to K$.}\footnote{This fact was first conjectured by Mima Stanojkovski (private e-mail communication).}

In turn, a remarkable result of Nikolov and Segal~\cite{NS2007a,NS2007b} states that the topological structure of a topologically finitely generated profinite group is completely determined by its algebraic structure. More precisely, the subgroups of finite index of a topologically finitely generated profinite group coincide with the open subgroups. (In the special case of pro-$p$-groups, this was proved by Serre in the 1970s, cf.\ \cite[Exercise~6 p.~32]{Serre2002}.) This implies that (i) any group homomorphism from a topologically finitely generated profinite group to a finite discrete group is continuous, and (ii) any two topologically finitely generated profinite groups are isomorphic as topological groups if, and only if, they are isomorphic as abstract groups. The homomorphism counting result in the previous paragraph can then be  restated as follows: \emph{Two topologically finitely generated profinite groups $G,H$ are isomorphic if and only if, for all finite groups $K$, the number of group homomorphisms $G\to K$ is the same as the number of group homomorphisms $H\to K$.}

A similar result, whose statement we shall omit, holds for (topologically finitely generated) Abelian profinite groups, also known as \emph{proabelian groups}, by taking as $\C$ the class of finite Abelian groups. 

For an example where $\C$ is a proper subclass of $\V_{\fin}$, let $\C$ consist of the finite $p$-groups. We deduce that: \emph{Two topologically finitely generated pro-$p$-groups $G,H$ are isomorphic if and only if, for all finite $p$-groups $K$, the number of group homomorphisms $G\to K$ is the same as the number of group homomorphisms $H\to K$.} 

\vspace{0.5em}
A typical homomorphism counting result will count the number of morphisms (into, or from a given object) in the monoid $\N$ of natural numbers. Now, recall that the proof of Theorem~\ref{th:Lovasz-finite-type} (and similarly, the proof of Theorem~\ref{thm:right-combinatorial-cats} in the locally finite case) consists essentially of two steps. First, using the Stirling kernel construction, we show that the functors $\M(-,a)$ and $\M(-,b)$ are pointwise isomorphic whenever the functors $\hom(-,a)$ and $\hom(-,b)$ are pointwise isomorphic. Second, Numakura's Lemma is invoked to show that $a\cong b$ if, for all $c$,
\[
\M(c,a)\neq\emptyset \enspace \Longleftrightarrow \enspace \M(c,b)\neq\emptyset.
\] 
Dispensing with the first step, we obtain a (weaker) result whereby the isomorphism type of an object is determined by counting the number of embeddings (or, dually, quotients) in the two-element monoid $(\{0,1\},\vee,0)$. We state a special case of this result, which generalises a known fact in profinite group theory (cf.\ \cite[Theorems~3.2.7 and~3.2.9]{RZ2010}).

\begin{proposition}
Let $\V$ be a variety of universal algebras, $\C\subseteq \V_{\fin}$ a class of finite algebras closed under subalgebras and finite products, and $A,B$ two topologically finitely generated pro-$\C$ algebras. Suppose that, for all (discrete) algebras $C\in \C$, there exists a quotient\footnote{That is, a continuous surjective homomorphism.} $A\epi C$ if and only if there exists a quotient $B\epi C$. Then $A\cong B$.
\end{proposition}

\subsection{Finite-variable logics}\label{s:finite-variable-logic}
In this section we assume familiarity with the basic notions of first-order logic and finite model theory. We shall present an application of the results in Section~\ref{s:beyond-loc-finite} to homomorphism counting in finite model theory; this is the topic of the recent work~\cite{DJR2021}. Following~\cite{DJR2021}, we apply the framework of \emph{game comonads} in (finite) model theory introduced by Abramsky, Dawar \emph{et al.} in~\cite{Abramsky2017b,AbramskyShah2018}. 

Let us fix a finite relational signature $\sigma$ and a positive integer $k$. We recall from~\cite{Abramsky2017b} the \emph{pebbling comonad} $\Pk$ on the category $\R(\sigma)$ of $\sigma$-structures, which models $k$-pebble games. 

Set $\k\coloneqq \{1,\dots,k\}$. Given a $\sigma$-structure $A$, we consider the set ${(\k\times A)^+}$ of \emph{plays} in $A$, i.e.\ all non-empty finite sequences of elements of $\k\times A$. A pair $(p,a) \in {\k\times A}$ is called a \emph{move}. Intuitively, the move $(p,a)$ corresponds to placing the pebble $p$ on the element $a$. Whenever $[(p_1,a_1),\dots,(p_l,a_l)]$ is a play, $p_i$ is called the \emph{pebble index} of the move $(p_i,a_i)$. Define the map
\[
\epsilon_A\colon (\k\times A)^+\to A, \ \ [(p_1,a_1),\dots,(p_l,a_l)] \mapsto a_l
\]
sending a play to the element of $A$ in its last move. Let $\Pk(A)$ be the $\sigma$-structure with universe $(\k\times A)^+$ and such that, for every $R\in\sigma$ of arity~$j$, its interpretation $R^{\Pk(A)}$ consists of the tuples of plays $(s_1,\ldots,s_j)$ such that: 
\begin{enumerate}[label=(\roman*)]
\item The $s_i$'s are pairwise comparable in the prefix order.
\item Whenever $s_i$ is a prefix of~$s_{i'}$, the pebble index of the last move in $s_i$ does not appear in the suffix of $s_i$ in $s_{i'}$.
\item ${(\epsilon_A(s_1),\dots,\epsilon_A(s_j))\in R^A}$.
\end{enumerate}
This assignment extends to a functor $\R(\sigma)\to\R(\sigma)$ by setting, for all homomorphisms of $\sigma$-structures $h\colon A\to B$,
\[
\Pk(h)\colon \Pk(A)\to \Pk(B), \ \ [(p_1,a_1),\dots,(p_l,a_l)] \mapsto [(p_1,h(a_1)),\dots,(p_l,h(a_l))].
\]

In fact, $\Pk$ is a comonad on $\R(\sigma)$ when equipped with the counit $\epsilon$ described above and the comultiplication $\delta_A \colon \Pk(A)\to \Pk\Pk(A)$ given by
\[
[(p_1,a_1),\dots,(p_l,a_l)]\mapsto [(p_1,s_1),\ldots,(p_l,s_l)],
\]
where $s_i\coloneqq [(p_1,a_1),\ldots,(p_i,a_i)]$ for $i\in\{1,\ldots,l\}$. 
Moreover, it follows directly from \cite[Proposition~22]{Abramsky2017b} that a $\sigma$-structure $A$ admits a coalgebra structure for $\Pk$ if, and only if, it has \emph{tree-width} less than $k$ (the notion of tree-width for relational structures was introduced in~\cite{FV1999} and generalises the homonymous concept for graphs). 

Denote by $\CL$ (respectively, $\CL(\woeq)$) the extension of first-order logic (respectively, first-order logic \emph{without equality}) obtained by adding \emph{counting quantifiers} $\exists_{{\geq}i}$ for all natural numbers $i$. The $k$-variable fragments of $\CL$ and $\CL(\woeq)$ are denoted by $\CL^k$ and $\CL^k(\woeq)$, respectively. 
Further, consider the adjunction $U\dashv F\colon \R(\sigma)\to\EM(\Pk)$ associated with the comonad $\Pk$. For any two finite $\sigma$-structures $A$ and $B$ we have
\[
F(A)\cong F(B) \enspace \Longleftrightarrow \enspace A\equiv_{\CL^k(\text{\scriptsize{$\overset{w.o.}=$}})} B,
\] 
i.e.\ $F(A)\cong F(B)$ precisely when $A$ and $B$ satisfy the same sentences of $\CL^k(\woeq)$. For the latter assertion, cf.\ \cite[Theorem~18]{Abramsky2017b} and \cite[\S VI]{DJR2021}.

It is well known that $\R(\sigma)$ is a locally finitely presentable category, and every finitely presentable $\sigma$-structure is finite. Since we assumed that the signature $\sigma$ is finite, the converse holds as well, and so the finitely generated objects in $\R(\sigma)$ coincide with the finitely presentable ones and are precisely the finite $\sigma$-structures. See e.g.\ \cite[pp.~200--201]{AR1994}.
It follows easily from the criterion in Lemma~\ref{l:sufficient-cond-finite-rank} in the appendix that $\Pk$ is finitary, and it can be verified that it satisfies the conditions in the definition of comonad of finite rank (cf.\ also Remark~\ref{rem:item-iii-monic}). As the forgetful functor $U\colon\EM(\Pk)\to \R(\sigma)$ preserves and reflects finitely presentable objects by Lemma~\ref{l:preserves-reflects-fp-objects}, a coalgebra $x\in\EM(\Pk)$ is finitely presentable if and only if $U(x)$ is a finite $\sigma$-structure.

Therefore, in view of Corollary~\ref{cor:hom-counting-comonads-finite-rank} and Remark~\ref{rem:replace-fg-with-fp}, the following statements are equivalent for all finite $\sigma$-structures $A,B$:
\begin{enumerate}
\item $A\equiv_{\CL^k(\text{\scriptsize{$\overset{w.o.}=$}})} B$.
\item For all finite $\sigma$-structures $C$ with tree-width less than $k$, the number of homomorphisms $C\to A$ is the same as the number of homomorphisms $C\to B$.
\end{enumerate}

To characterise equivalence in the logic $\CL^k$ \emph{with equality}, we proceed as follows. Let $\sigma'$ be the relational signature obtained by adding a binary relation symbol $I$ to $\sigma$. There is an adjunction ${H\dashv J\colon \R(\sigma)\to \R(\sigma')}$, where:
\begin{itemize}
\item $J$ sends a $\sigma$-structure $A$ to the $\sigma'$-structure obtained by interpreting $I$ as the identity relation on $A$;
\item $H$ sends a $\sigma'$-structure $B$ to the quotient structure $B^-/{\sim}$, where $B^-$ is the $\sigma$-reduct of $B$ and $\sim$ is the equivalence relation generated by the interpretation of $I$ in $B$.
\end{itemize}
For a proof of this fact, see \cite[Lemma~25]{DJR2021}. 

Since the comonad $\Pk$ was defined for an arbitrary relational signature $\sigma$, we have a corresponding comonad $\Pk'$ on $\R(\sigma')$ with associated adjunction $U'\dashv F'$. It follows from \cite[Theorem~18]{Abramsky2017b} that, for all finite $\sigma$-structures~$A,B$, 
\[
F'J(A)\cong F'J(B) \enspace \Longleftrightarrow \enspace A\equiv_{\CL^k} B.
\]
Moreover, the adjunction $H\dashv J$ restricts to the full subcategories of $\R(\sigma)$ and $\R(\sigma')$ defined by the structures with tree-width less than $k$ (this assertion is trivial for $J$; for $H$, this follows directly from \cite[Proposition~23]{DJR2021}).

Thus, Corollary~\ref{cor:hom-counting-comonads-finite-rank-relative} and Remark~\ref{rem:replace-fg-with-fp} entail the following result, which was first proved in \cite[Theorem~21]{DJR2021} for an arbitrary relational signature $\sigma$ (by means of an indirect argument) and generalises a result of Dvo\v{r}\'{a}k for graphs~\cite{dvovrak2010recognizing}:
\begin{theorem}
Let $\sigma$ be a finite relational signature and let $A,B$ be any two finite $\sigma$-structures. Then $A\equiv_{\CL^k} B$ if and only if, for all finite $\sigma$-structures $C$ with tree-width less than $k$, the number of homomorphisms $C\to A$ is the same as the number of homomorphisms $C\to B$.
\end{theorem}

%%%%%%%%%%%%%%%%%%%%%%%%%%%%%%%%%%%%%%%%%%%%%
%%%%%%%%%%%%%%%%%%%%%%%%%%%%%%%%%%%%%%%%%%%%%
\appendix

\section{Proper Factorisation Systems}\label{s:fact-systems}

In this section we recall the notion of proper factorisation system in a category $\A$, and some of its main properties.

Given arrows $e$ and $m$ in $\A$, we say that $e$ has the \emph{left lifting property} with respect to $m$, or that $m$ has the \emph{right lifting property} with respect to $e$, if for every commutative square as on the left-hand side below
\begin{equation*}
\begin{tikzcd}
{\cdot} \arrow{d} \arrow{r}{e} & {\cdot} \arrow{d} \\
{\cdot} \arrow{r}{m} & {\cdot}
\end{tikzcd}
\ \ \ \ \ \ \ \ \ \ \ \ \ 
\begin{tikzcd}
{\cdot} \arrow{d} \arrow{r}{e} & {\cdot} \arrow{d} \arrow[dashed]{dl}[swap, outer sep=-1pt]{d} \\
{\cdot} \arrow{r}{m} & {\cdot}
\end{tikzcd}
\end{equation*}
there exists a (not necessarily unique) \emph{diagonal filler}, i.e.\ an arrow $d$ such that the right-hand diagram above commutes. If this is the case, we write $e {\,\pit\,} m$. For any class $\mathscr{H}$ of morphisms in $\A$, let ${}^{\pit}\mathscr{H}$ (respectively $\mathscr{H}^{\pit}$) be the class of morphisms having the left (respectively right) lifting property with respect to every morphism in $\mathscr{H}$.

\begin{definition}\label{def:weak-f-s}
A pair of classes of morphisms $(\Q,\M)$ in a category $\A$ is a \emph{weak factorisation system} provided it satisfies the following conditions:
\begin{enumerate}[label=(\roman*)]
\item Every arrow $f$ in $\A$ decomposes as $f = m \circ e$ with $e\in \Q$ and $m\in \M$.
\item $\Q={}^{\pit}\M$ and $\M=\Q^{\pit}$.
\end{enumerate}
A \emph{proper factorisation system} is a weak factorisation system $(\Q,\M)$ such that all arrows in $\Q$ are epimorphisms and all arrows in $\M$ are monomorphisms. 

We refer to $\Q$-morphisms and $\M$-morphisms, respectively, as \emph{quotients} (denoted by $\epi$) and \emph{embeddings} (denoted by $\emb$).
\end{definition}

It is easy to see that any proper factorisation system is an \emph{orthogonal} factorisation system, meaning that the diagonal fillers are unique. In particular, factorisations are unique up to (unique) isomorphism. 

Furthermore, recall that a \emph{strong epimorphism} is an epimorphism that has the left lifting property with respect to all monomorphisms, and a \emph{strong monomorphism} is a monomorphism that has the right lifting property with respect to all epimorphisms. Given any proper factorisation system $(\Q,\M)$, the following inclusions are immediate:
\begin{gather*}
\{\text{strong monomorphisms}\}\subseteq \M \subseteq \{\text{monomorphisms}\}, \\
\{\text{strong epimorphisms}\}\subseteq \Q \subseteq \{\text{epimorphisms}\}.
\end{gather*}

Next, we state some well known properties of weak factorisation systems (cf.\ \cite{freyd1972categories} or~\cite{riehl2008factorization}):
\begin{lemma}\label{l:factorisation-properties}
Let $(\Q,\M)$ be a weak factorisation system in $\A$. The following statements hold:
\begin{enumerate}[label=(\alph*)]
\item\label{compositions} $\Q$ and $\M$ are closed under compositions.
\item\label{isos} $\Q\cap\M=\{\text{isomorphisms}\}$.
\item\label{pullbacks} The pullback of an embedding along any morphism, if it exists, is again an embedding.
\item\label{pushouts} The pushout of a quotient along any morphism, if it exists, is again a quotient.
\end{enumerate}
Moreover, if $(\Q,\M)$ is proper, then the following hold:
\begin{enumerate}[label=(\alph*)]\setcounter{enumi}{4}
\item\label{cancellation-e} $g\circ f\in \Q$ implies $g\in\Q$.
\item\label{cancellation-m} $g\circ f\in\M$ implies $f\in\M$.
\end{enumerate}
\end{lemma}

Let $\A$ be a category equipped with a proper factorisation system $(\Q,\M)$. In the same way that one usually defines the poset of subobjects of a given object $a\in\A$, we can define the poset of $\M$-subobjects of $a$. Given embeddings $m\colon b\emb a$ and $n\colon c\emb a$, let us say that $m\trianglelefteq n$ provided there is a morphism $i\colon b\to c$ such that $m=n\circ i$ (note that, if it exists, $i$ is necessarily an embedding). 
This yields a preorder on the class of all embeddings with codomain $a$. The symmetrization $\sim$ of $\trianglelefteq$ can be characterised as follows: $m\sim n$ if, and only if, $m=n\circ i$ for some isomorphism $i\colon b\to c$. 

Let $\M(a)$ be the class of $\sim$-equivalence classes of embeddings with codomain $a$, equipped with the natural partial order $\leq$ induced by~$\trianglelefteq$. We systematically represent a $\sim$-equivalence class by any of its representatives. As every embedding is a monomorphism, and we are assuming that all categories under consideration are well-powered (cf.\ Assumption~\ref{assump:well-(co)powered}), we see that $\M(a)$ is a set. We refer to $\M(a)$ as the \emph{poset of embeddings} of $a$. Likewise, dualising the construction outlined above, we consider the \emph{poset of quotients} of $a$, denoted by $\Q(a)$. Hence, for any two (equivalence classes of) quotients $f\colon a\epi b$ and $g\colon a\epi c$, we have $f\leq g$ if and only if $f$ factors through $g$.

\begin{remark}\label{rem:dual-proper-f-s}
Let $(\Q,\M)$ be a proper factorisation system in $\A$. Recall from Section~\ref{s:polyad-spaces-SK} that $\A$ is $\Q$-well-founded (respectively, $\M$-well-founded) if, for all objects $a\in \A$, the poset $\Q(a)$ (respectively, $\M(a)$) is well-founded. 

Let $\Q'$ be the class of morphisms in $\A^\op$ obtained by reversing the arrows in $\M$, and $\M'$ the class of morphisms in $\A^\op$ obtained by reversing the arrows in $\Q$. Then $(\Q',\M')$ is a proper factorisation system in~$\A^\op$, and $\A$ is $\Q$-well-founded if and only if $\A^\op$ is $\M'$-well-founded. Just observe that the posets $\Q(a)$ and $\M'(a)$ are isomorphic for all objects $a$ (because, intuitively, the notion of `factoring through' is self-dual).
\end{remark}

\begin{lemma}\label{l:amalg-quot}
Let $\A$ be a category equipped with a proper factorisation system $(\Q,\M)$. The following statements hold:
\begin{enumerate}[label=(\alph*)]
\item\label{comma-cat-quot} Any commutative triangle in $\A$ admits a commutative subdivision as displayed below.
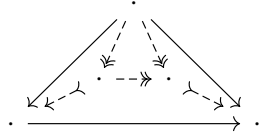
\[\begin{tikzcd}[column sep=0.7, row sep=4]
{} & {} & {} & {} & {\cdot} \arrow{dddllll} \arrow{dddrrrr} \arrow[twoheadrightarrow,dashed,shorten >=-0.5ex]{ddl} \arrow[twoheadrightarrow,dashed,shorten >=-0.5ex]{ddr} & {} & {} & {} & {} \\
{} & {} & {} & {} & {} & {} & {} & {} & {} \\
{} & {} & {} & {\cdot} \arrow[twoheadrightarrow,dashed]{rr} \arrow[rightarrowtail,dashed,shorten >=1.5ex]{dlll} & {} & {\cdot} \arrow[rightarrowtail,dashed,shorten >=1.5ex]{drrr} & {} & {} & \\
{\cdot} \arrow{rrrrrrrr} & {} & {} & {} & {} & {} & {} & {} & {\cdot}
\end{tikzcd}\]
\item\label{span-of-quot-compl} If $\A$ has the amalgamation property, then any span of the form 
$\begin{tikzcd}[column sep=1.5em]
{\cdot} & {\cdot} \arrow{l} \arrow[twoheadrightarrow]{r} & {\cdot} 
\end{tikzcd}$ 
can be completed to a commutative square as follows:
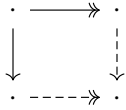
\[\begin{tikzcd}
{\cdot} \arrow[twoheadrightarrow]{r} \arrow{d} & {\cdot} \arrow[dashed]{d} \\
{\cdot} \arrow[dashed,twoheadrightarrow]{r} & {\cdot}
\end{tikzcd}\]
In particular, any span of quotients in $\A$ can be completed to a commutative square of quotients.
\end{enumerate}
\end{lemma}
\begin{proof}
For item (a), consider morphisms $f_1,f_2,g$ in $\A$ such that ${g\circ f_1=f_2}$. Let $(e_1,m_1)$ and $(e_2,m_2)$ be the $(\Q,\M)$ factorisations of $f_1$ and $f_2$, respectively. We obtain a commutative diagram as follows.
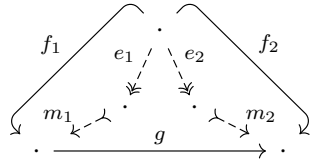
\[\begin{tikzcd}[column sep=0.7, row sep=4]
{} & {} & {} & {} & {\cdot} \arrow[dddllll, relay arrow=-1.5ex, swap, "f_1"]  \arrow[dddrrrr, relay arrow=1.5ex, "f_2"] \arrow[twoheadrightarrow,dashed,swap,shorten >=-0.5ex]{ddl}{e_1}\arrow[twoheadrightarrow,dashed,shorten >=-0.5ex]{ddr}{e_2} & {} & {} & {} & {} \\
{} & {} & {} & {} & {} & {} & {} & {} & {} \\
{} & {} & {} & {\cdot} \arrow[rightarrowtail,dashed,swap,shorten >=1.5ex]{dlll}{m_1} & {} & {\cdot} \arrow[rightarrowtail,dashed,shorten >=1.5ex]{drrr}{m_2} & {} & {} & \\
{\cdot} \arrow{rrrrrrrr}{g} & {} & {} & {} & {} & {} & {} & {} & {\cdot}
\end{tikzcd}\]
Then the following square admits a diagonal filler,
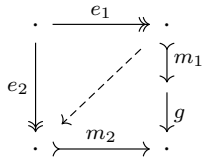
\[\begin{tikzcd}[column sep=1.7em, row sep=1.2em]
{\cdot} \arrow[twoheadrightarrow]{dd}[swap]{e_2} \arrow[twoheadrightarrow]{rr}{e_1} & & {\cdot} \arrow[rightarrowtail,shorten >=-0.5ex]{d}{m_1} \arrow[dashed,shorten >=1.0ex,shorten <=1.0ex]{ddll} \\
& & \arrow[shorten <=-0.5ex]{d}{g} \\
{\cdot} \arrow[rightarrowtail]{rr}{m_2} & & {\cdot}
\end{tikzcd}\]
which is a quotient by Lemma~\ref{l:factorisation-properties}\ref{cancellation-e}. This settles the statement.

For item (b), suppose we have a span 
$\begin{tikzcd}[column sep=1.5em]
{\cdot} & {\cdot} \arrow{l} \arrow[twoheadrightarrow]{r} & {\cdot} 
\end{tikzcd}$ 
in $\A$. As $\A$ has the amalgamation property, we can complete this span to a commutative square as on the left-hand side below.
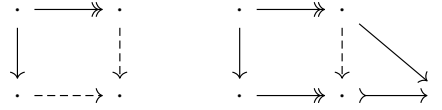
\begin{center}
\begin{tikzcd}
{\cdot} \arrow[twoheadrightarrow]{r} \arrow{d} & {\cdot} \arrow[dashed]{d} \\
{\cdot} \arrow[dashed]{r} & {\cdot}
\end{tikzcd}
\ \ \ \ \ \ \ \ 
\begin{tikzcd}
{\cdot} \arrow[twoheadrightarrow]{r} \arrow{d} & {\cdot} \arrow[dashed]{d} \arrow{dr} & \\
{\cdot} \arrow[twoheadrightarrow]{r} & {\cdot} \arrow[rightarrowtail]{r} & {\cdot}
\end{tikzcd}
\end{center}
If we consider the $(\Q,\M)$ factorisation of the lower horizontal morphism, then we obtain a diagonal filler as on the right-hand side above. This yields the desired completion of the original span. For the second part of the statement, note that in any commuting square of the form
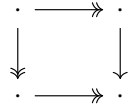
\[\begin{tikzcd}
{\cdot} \arrow[twoheadrightarrow]{r} \arrow[twoheadrightarrow]{d} & {\cdot} \arrow{d} \\
{\cdot} \arrow[twoheadrightarrow]{r} & {\cdot}
\end{tikzcd}\]
the right vertical morphism must be a quotient by Lemma~\ref{l:factorisation-properties}\ref{compositions},\ref{cancellation-e}.
\end{proof}

Let $F\colon \A^\op\to \Set$ be a presheaf, and let $U\colon \int F \to \A$ be the natural forgetful functor defined on the category of elements of $F$. If $\A$ is equipped with a proper factorisation system $(\Q,\M)$, then we can equip $\int F$ with a factorisation system $(\Q',\M')$ by simply letting $\Q'$ be the class of morphisms $f$ in $\int F$ such that $U(f)\in \Q$, and $\M'$ the class of morphisms $g$ in $\int F$ such that $U(g)\in\M$. It is not difficult to see that $(\Q',\M')$ is a proper factorisation system. In the proof of the next lemma, we will assume that $\int F$ is equipped with this factorisation system.

\begin{lemma}\label{l:amalg-quotients}
Let $\A$ be a category equipped with a proper factorisation system, and let $F\colon \A^\op \to\Set$ be a polyadic set. For all spans 
\[\begin{tikzcd}
b & {\cdot} \arrow{l}[swap]{g} \arrow[twoheadrightarrow]{r}{f} & a
\end{tikzcd}\]
in $\A$ and elements $x\in F(a)$ and $y\in F(b)$ such that $F(f)(x)=F(g)(y)$, there exist a commutative square
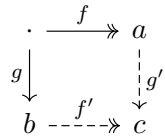
\[\begin{tikzcd}
{\cdot} \arrow[twoheadrightarrow]{r}{f} \arrow{d}[swap]{g} & a \arrow[dashed]{d}{g'} \\
b \arrow[dashed,twoheadrightarrow]{r}{f'} & c
\end{tikzcd}\]
in $\A$ and $z\in F(c)$ such that $F(g')(z)=x$ and $F(f')(z)=y$. Moreover, if $g$ is a quotient, we can assume that $g'$ is also a quotient.
\end{lemma}
\begin{proof}
This follows directly from Lemma~\ref{l:amalg-quot}\ref{span-of-quot-compl} by setting $\A\coloneqq \int F$.
\end{proof}

\section{Comonads of Finite Rank}\label{s:finite-rank}

Given a monad on a locally finitely presentable category, we may ask if the associated category of Eilenberg-Moore algebras is locally finitely presentable. The answer is positive, provided the monad is finitary (see e.g.~\cite[p.~124]{AR1994}). The analogous question for \emph{comonads} requires more care and leads to the notion of comonad of finite rank (Definition~\ref{def:finite-rank} below). 
In this section we shall collect some basic facts about comonads of finite rank; most of the material presented here can be found in Diers' monograph~\cite{Diers1986}, albeit in a more condensed form (detailed references are provided below).

Let $(T,\delta,\epsilon)$ be a comonad on a category $\A$, with comultiplication $\delta$ and counit $\epsilon$. Recall that an \emph{Eilenberg-Moore coalgebra} for $T$ is a pair $(a,g)$ where $a$ is an object of $\A$ and $g\colon a\to T(a)$ is a morphism in $\A$ such that the following diagrams commute.
\begin{equation*}
\begin{tikzcd}
a \arrow{r}{g} \arrow[swap]{d}{g} & T(a) \arrow{d}{\delta_a} \\
T(a) \arrow{r}{T(g)} & T^2 (a)
\end{tikzcd}
\ \ \ \ \ \ 
\begin{tikzcd}
a \arrow{r}{g} \arrow{dr}[swap]{\id_a} & T(a) \arrow{d}{\epsilon_a} \\
{} & a
\end{tikzcd}
\end{equation*}
A morphism of coalgebras $(a,g)\to(b,h)$ is an arrow $f\colon a\to b$ in $\A$ satisfying $T(f)\circ g=h\circ f$. Eilenberg-Moore coalgebras and their morphisms form a category, denoted by $\EM(T)$. The forgetful functor $U\colon \EM(T)\to\A$, sending a coalgebra morphism $(a,g)\to(b,h)$ to the corresponding arrow $a\to b$, has a right adjoint $F\colon \A\to\EM(T)$ which sends an arrow $f\colon a\to b$ in $\A$ to $T(f)\colon (T(a), \delta_a)\to (T(b),\delta_b)$. We refer to $U\dashv F$ as the adjunction canonically associated with $T$.

Recall that a functor is \emph{finitary} if it preserves directed colimits. By extension, we say that a comonad is finitary if so is its underlying functor. 

The next lemma, which provides sufficient conditions for a comonad to be finitary, is an immediate consequence of a result of Ad\'{a}mek, Milius, Sousa and Wissmann \cite[Theorem~3.4]{AMSW2019}. Note that item~(ii) below is a direct generalisation of the `$\epsilon$-$\delta$ characterisation' of continuity for Scott continuous maps between algebraic domains, see e.g.\ \cite[Proposition~II-2.1]{GHKLMS2003} (we are grateful to Samson Abramsky for pointing this out to us).

\begin{lemma}\label{l:sufficient-cond-finite-rank}
Let $\A$ be a locally finitely presentable category in which every finitely generated object is finitely presentable, and let $T$ be a comonad on $\A$ that preserves monomorphisms. The following statements are equivalent:
\begin{enumerate}[label=(\roman*)]
\item $T$ is finitary.
\item For all $a\in \A$, every subobject $b\emb T(a)$ with $b$ finitely presentable factors through the $T$-image of a finitely presentable subobject of $a$.
\end{enumerate}
\end{lemma}

The following useful fact was proved by Diers in \cite[p.~45]{Diers1986}, where some details were omitted; for the sake of completeness, we offer a detailed proof.

\begin{lemma}\label{l:preserves-reflects-fp-objects}
Let $T$ be a finitary comonad on a category $\A$, with canonically associated adjunction $U\dashv F\colon \A\to \EM(T)$. The following statements hold:
\begin{enumerate}[label=(\alph*)]
\item\label{co-free-directed-colim} $F$ preserves directed colimits.
\item\label{forgetful-fp} $U$ preserves and reflects finitely presentable objects.
\end{enumerate}
\end{lemma}
\begin{proof}
For item~(a), let $\{c_i\mid i\in I\}$ be an arbitrary directed diagram in $\A$. Since $T$ is finitary and $U$ preserves colimits (as it is left adjoint),
\[
UF(\colim_{i\in I}a_i)\cong \colim_{i\in I} UF(a_i)\cong U(\colim_{i\in I}F(a_i)).
\]
That is, if $\phi\colon \colim_{i\in I}F(a_i)\to F(\colim_{i\in I}a_i)$ denotes the obvious mediating morphism, then $U(\phi)$ is an isomorphism.
Using the fact that $U$ reflects isomorphisms (see e.g.\ \cite[Proposition~20.12(5)]{AHS1990} for a proof of the dual fact), we deduce that $F(\colim_{i\in I}a_i)\cong \colim_{i\in I}F(a_i)$.

For item~(b), consider an arbitrary coalgebra $x=(b,h)\in \EM(T)$. If $x$ is finitely presentable then, for any directed diagram $\{a_i\mid i\in I\}$ in $\A$,
\begin{align*}
\hom_{\A}(U(x),\colim_{i\in I}a_i)&\cong \hom_{\EM(T)}(x, F(\colim_{i\in I}a_i)) \\
&\cong \hom_{\EM(T)}(x, \colim_{i\in I}F(a_i)) \tag*{by~(a)}\\
&\cong \colim_{i\in I} \hom_{\EM(T)}(x, F(a_i)) \\
&\cong \colim_{i\in I} \hom_{\A}(U(x), a_i).
\end{align*}
Thus, $U$ preserves finitely presentable objects.

To see that $U$ reflects finitely presentable objects, assume that $b=U(x)$ is finitely presentable. Let $\{(a_i,g_i)\mid i\in I\}$ be a directed diagram in $\EM(T)$ and $\alpha_i\colon (a_i,g_i)\to (a,g)$ a colimit cocone. We claim that the function
\begin{equation}\label{eq:comparison-homsets}
\colim_{i\in I}\hom_{\EM(T)}((b,h),(a_i,g_i)) \ \longrightarrow \ \hom_{\EM(T)}((b,h),(a,g))
\end{equation}
induced by the cocone
\[
\alpha_i\circ -\colon \hom_{\EM(T)}((b,h),(a_i,g_i)) \ \longrightarrow \ \hom_{\EM(T)}((b,h),(a,g))
\] 
is bijective, thus showing that $x=(b,h)$ is finitely presentable.

Suppose that $f_i\colon (b,h)\to (a_i,g_i)$ and $f_j\colon (b,h)\to (a_j,g_j)$ are morphisms in $\EM(T)$ such that $\alpha_i \circ f_i=\alpha_j \circ f_j$. As the diagram $I$ is directed, there are transition morphisms $\phi_{ik}\colon (a_i,g_i)\to (a_k,g_k)$ and $\phi_{jk}\colon (a_j,g_j)\to (a_k,g_k)$ such that 
\[
\alpha_k \circ \phi_{ik}=\alpha_i \ \text{ and } \ \alpha_k \circ \phi_{jk}=\alpha_j.
\] 
Hence $\alpha_k\circ \phi_{ik}\circ f_i=\alpha_k \circ \phi_{jk} \circ f_j$. By the definition of finitely presentable object, there exists a transition morphism $\phi_{kk'}$ such that $\phi_{ik'}\circ f_i=\phi_{jk'}\circ f_j$. That is, $f_i=f_j$ in $\colim_{i\in I}\hom_{\EM(T)}((b,h),(a_i,g_i))$.
Therefore, the map in~\eqref{eq:comparison-homsets} is injective.

For surjectivity, consider an arbitrary morphism $f\colon (b,h)\to (a,g)$ in $\EM(T)$. Because $U$ preserves colimits, $U(f)$ belongs to 
\[
\hom_{\A}(b,a)\cong \hom_{\A}(b,\colim_{i\in I}a_i)\cong \colim_{i\in I} \hom_{\A}(b, a_i),
\]
so there exist $i\in I$ and $f_i\colon b\to a_i$ such that $U(f)=U(\alpha_i) \circ f_i$.
\begin{claim*}
$TU(\alpha_i) \circ g_i \circ f_i = TU(\alpha_i) \circ T(f_i) \circ h$.
\end{claim*}
\begin{proof}[Proof of Claim]
The commutativity of the diagrams
\begin{equation*}
\begin{tikzcd}
b \arrow{r}{f_i} \arrow{dr}[swap]{U(f)} & a_i \arrow{d}{U(\alpha_i)} \arrow{r}{g_i} & T(a_i) \arrow{d}{TU(\alpha_i)} \\
{} & a \arrow{r}{g} & T(a)
\end{tikzcd}
\quad \quad \quad \quad 
\begin{tikzcd}
b \arrow{r}{h} \arrow{d}[swap]{U(f)} & T(b) \arrow{d}{TU(f)} \\
a \arrow{r}{g} & T(a)
\end{tikzcd}
\end{equation*}
entails that $TU(\alpha_i) \circ g_i \circ f_i = TU(f) \circ h$. Since $U(f)=U(\alpha_i) \circ f_i$, we obtain
\[
TU(\alpha_i) \circ g_i \circ f_i = T(U(\alpha_i) \circ f_i) \circ h = TU(\alpha_i) \circ T(f_i) \circ h.\qedhere
\]
\end{proof}
\noindent As $T$ preserves directed colimits, $T(a)$ is the colimit of the cocone given by the morphisms $TU(\alpha_i)\colon T(a_i)\to T(a)$. By the Claim and the definition of finitely presentable object, there is a transition morphism $U(\phi_{ij})\colon a_i\to a_j$ satisfying $TU(\phi_{ij})\circ g_i \circ f_i = TU(\phi_{ij}) \circ T(f_i) \circ h$. Because $TU(\phi_{ij})\circ g_i=g_j \circ U(\phi_{ij})$, we obtain
\[
g_j \circ U(\phi_{ij}) \circ f_i = T(U(\phi_{ij}) \circ f_i) \circ h,
\]
showing that $U(\phi_{ij})\circ f_i\colon (b,h)\to (a_j,g_j)$ is a morphism in $\EM(T)$.
Furthermore, note that the following diagram commutes
\[\begin{tikzcd}
b \arrow{dr}[swap]{U(\phi_{ij})\circ f_i} \arrow{rr}{U(f)} & & a \\
{} & a_j \arrow{ur}[swap]{U(\alpha_j)} &
\end{tikzcd}\]
and so the function in~\eqref{eq:comparison-homsets} is surjective.
\end{proof}

\begin{definition}\label{def:finite-rank}
Let $T$ be a comonad on a category $\A$, with forgetful functor $U\colon \EM(T)\to \A$. We say that $T$ is of \emph{finite rank} provided that the following conditions are satisfied:
\begin{enumerate}[label=(\roman*)]
\item $T$ is finitary.
\item Every morphism $f\colon a\to U(x)$ in $\A$, with $a$ finitely presentable, admits a factorisation
\[\begin{tikzcd}[column sep=3em]
a \arrow[rr, relay arrow=3ex, "f"] \arrow[r, "\beta"] & U(y) \arrow[r, "U(\gamma)"] & U(x)
    \end{tikzcd}\]
where $U(y)$ is finitely presentable.
\item The factorisation in~(ii) is essentially unique, i.e., if $\beta'\colon a\to U(y)$ satisfies $U(\gamma) \circ \beta'=f$, then there exist $y'\in \EM(T)$ with $U(y')$ finitely presentable and morphisms 
\[\begin{tikzcd}
y \arrow{r}{\lambda} & y' \arrow{r}{\gamma'} & x
\end{tikzcd}\]
in $\EM(T)$ such that $U(\lambda)\circ \beta=U(\lambda)\circ \beta'$ and $\gamma=\gamma'\circ \lambda$.
\end{enumerate} 
\end{definition}

\begin{remark}\label{rem:item-iii-monic}
Note that, if the morphism $U(\gamma)$ in item~(ii) is a monomorphism, then item~(iii) is automatically satisfied.
\end{remark}

Comonads of finite rank were introduced by Diers in \cite[Definition~1.12.0]{Diers1986}, where item~(iii) in Definition~\ref{def:finite-rank} was omitted. This condition of essential uniqueness is needed to prove Theorem~\ref{t:EM-lfp} below, which appears as Proposition~1.12.1 in \emph{op.\ cit}. To remedy this omission, we offer a complete proof of the following result:

\begin{theorem}\label{t:EM-lfp}
Let $T$ be a comonad on a locally finitely presentable category $\A$. If $T$ is of finite rank then the Eilenberg-Moore category $\EM(T)$ is locally finitely presentable and the forgetful functor $U\colon \EM(T)\to\A$ preserves and reflects finitely presentable objects.
\end{theorem}
\begin{proof}
Note that $\EM(T)$ is cocomplete because $\A$ is cocomplete and $U$ creates colimits (cf.\ \cite[Proposition~20.12(10)]{AHS1990} for a proof of the dual fact).
Further, $U$ preserves and reflects finitely presentable objects by Lemma~\ref{l:preserves-reflects-fp-objects}\ref{forgetful-fp}. Since in $\A$ there is, up to isomorphism, only a set of finitely presentable objects, the same holds in $\EM(T)$. It remains to prove that every object of $\EM(T)$ is a directed colimit of finitely presentable objects. 

We will show that each $(a,g)\in\EM(T)$ is the colimit of the canonical (directed) diagram given by the forgetful functor 
\[
D\colon\EM(T)_{\fp}\down(a,g)\to \EM(T),
\]
where $\EM(T)_{\fp}$ is the full subcategory of $\EM(T)$ defined by the finitely presentable objects. 
Because $U$ reflects colimits, it is enough to show that $U(a,g)$ is the colimit of the diagram $U\circ D\colon \EM(T)_{\fp}\down(a,g)\to \A$. Further, since $U$ preserves finitely presentable objects, $U\circ D$ factors as
\[\begin{tikzcd}
\EM(T)_{\fp}\down(a,g) \arrow[rr, relay arrow=2ex, "U\circ D"] \arrow{r}{J} & \A_{\fp}\down U(a,g) \arrow{r}{D'} & \A
\end{tikzcd}\]
where $D'$ is the canonical diagram associated with $U(a,g)$ and $J$ is the functor sending a morphism $(b_1,h_1)\to (b_2,h_2)$ in $\EM(T)_{\fp}\down(a,g)$ to its image under $U$. Because the colimit of $D'$ is $U(a,g)$, it suffices to show that the functor $J$ is cofinal. We prove that the following two properties hold, which in turn imply that $J$ is cofinal (cf.\ \cite[Proposition~2.11.2]{Borceux1}):
\begin{itemize}[leftmargin=2em]
\item For all $c\in \A_{\fp}\down U(a,g)$ there is $c\to J(d)$ with $d\in \EM(T)_{\fp}\down(a,g)$.
\item For any two morphisms $\beta_1\colon c\to J(d_1)$ and $\beta_2\colon c\to J(d_2)$ with $c\in \A_{\fp}\down U(a,g)$ and $d_1,d_2\in \EM(T)_{\fp}\down(a,g)$, there are $d\in\EM(T)_{\fp}\down(a,g)$ and morphisms $\phi_1\colon d_1\to d$, $\phi_2\colon d_2\to d$ in $\EM(T)_{\fp}\down(a,g)$ satisfying $J(\phi_1)\circ \beta_1=J(\phi_2)\circ \beta_2$.
\end{itemize}

The first property is precisely item~(ii) in the definition of a comonad of finite rank. 
For the second one, consider any two objects $\gamma_1\colon (b_1,h_1)\to (a,g)$ and $\gamma_2\colon (b_2,h_2)\to (a,g)$ of $\EM(T)_{\fp}\down(a,g)$, an object $\alpha\colon c\to U(a,g)$ of $\A_{\fp}\down U(a,g)$, and arrows $\beta_1\colon c\to U(b_1,h_1)$ and $\beta_2\colon c\to U(b_2,h_2)$ such that 
\[
U(\gamma_1)\circ \beta_1=\alpha=U(\gamma_2)\circ \beta_2.
\]
Since $D$ is directed, there exist an object $\gamma\colon (b,h) \to (a,g)$ in $\EM(T)_{\fp}\down(a,g)$ and morphisms $\psi_1\colon (b_1,h_1)\to (b,h)$ and $\psi_2\colon (b_2,h_2)\to (b,h)$ in $\EM(T)$ such that $\gamma\circ \psi_1=\gamma_1$ and $\gamma\circ \psi_2=\gamma_2$. 
It follows that 
\[
U(\gamma)\circ U(\psi_1)\circ \beta_1 = U(\gamma_1)\circ \beta_1 = U(\gamma_2)\circ \beta_2 = U(\gamma)\circ U(\psi_2)\circ \beta_2.
\]
By item~(iii) in the definition of a comonad of finite rank, and the fact that $U$ reflects finitely presentable objects, there exist $\gamma'\colon (b',h')\to (a,g)$ in $\EM(T)_{\fp}\down(a,g)$ and $\lambda\in\hom_{\EM(T)}((b,h),(b',h'))$ satisfying $\gamma=\gamma'\circ \lambda$ and
\[
U(\lambda)\circ U(\psi_1)\circ \beta_1=U(\lambda)\circ U(\psi_2)\circ \beta_2.
\]
Therefore, the morphisms $\phi_1\coloneqq \lambda\circ \psi_1$ and $\phi_2\coloneqq \lambda\circ \psi_2$ satisfy $J(\phi_1)\circ \beta_1=J(\phi_2)\circ \beta_2$ as desired.
\end{proof}

\begin{corollary}\label{cor:fp-equal-fg}
Let $T$ be a comonad of finite rank on a locally finitely presentable category $\A$. If every finitely generated object in $\A$ is finitely presentable, then the same holds in $\EM(T)$.
\end{corollary}
\begin{proof}
Suppose that every finitely generated object in $\A$ is finitely presentable, and let $x$ be an arbitrary finitely generated object of $\EM(T)$. By Theorem~\ref{t:EM-lfp}, the category $\EM(T)$ is locally finitely presentable and the forgetful functor $U\colon \EM(T)\to\A$ preserves and reflects finitely presentable objects. Thus there is a strong epimorphism $f\colon y\epi x$ in $\EM(T)$ with $y$ finitely presentable. As left adjoint functors preserve strong epimorphisms (see e.g.\ \cite[Proposition~4.3.9]{Borceux1}), $U(f)$ is a strong epimorphism in $\A$. Since $U$ preserves finitely presentable objects, it follows that $U(x)$ is finitely generated in $\A$ and thus finitely presentable. Using the fact that $U$ reflects finitely presentable objects, we see that $x$ is finitely presentable.
\end{proof}

\addtocontents{toc}{\protect\setcounter{tocdepth}{1}}
\vspace{0.5em}
\subsection*{Acknowledgements} We are grateful to Samson Abramsky and Tom\'{a}\v{s} Jakl for several useful comments on a preliminary version of this paper that helped us to improve the presentation of our results.

%%
%% Bibliography
%%

\Urlmuskip=0mu plus 1mu\relax
\bibliographystyle{amsplain-nodash}

\providecommand{\bysame}{\leavevmode\hbox to3em{\hrulefill}\thinspace}
\providecommand{\MR}{\relax\ifhmode\unskip\space\fi MR }
% \MRhref is called by the amsart/book/proc definition of \MR.
\providecommand{\MRhref}[2]{%
  \href{http://www.ams.org/mathscinet-getitem?mr=#1}{#2}
}
\providecommand{\href}[2]{#2}

\end{document}